\documentclass[12pt,twoside]{amsart}
\usepackage{geometry}
\usepackage{amssymb,amsmath,amsthm, amscd, enumerate, mathrsfs}
\usepackage{graphicx, hhline}
\usepackage[all]{xy}
\usepackage[dvipdfmx]{hyperref}
\usepackage[shortlabels]{enumitem}
\setlist[enumerate]{leftmargin=56pt,labelsep=
8pt,itemsep=4pt,label=\upshape{(\thethm.\arabic*)}}
\usepackage{mathtools}
\mathtoolsset{showonlyrefs}
\usepackage{color}

\newcommand{\RomII}{\uppercase\expandafter{\romannumeral 2}}

\title[Variation of mixed Hodge structure]
{Variation of mixed Hodge structure and its applications}
\author{Osamu Fujino and Taro Fujisawa}
\date{2025/3/11, version 0.55}
\subjclass[2020]{Primary 32L20; Secondary 14D07, 14E30}
\keywords{Hodge bundles, 
vanishing theorems, strict support condition, 
torsion-freeness, injectivity theorem, variation of mixed 
Hodge structure, semipositivity theorem, minimal model program}
\address{Department of 
Mathematics, Graduate School of Science, 
Kyoto University, Kyoto 606-8502, Japan}
\email{fujino@math.kyoto-u.ac.jp}
\address{Department of Mathematics, School of Engineering, 
Tokyo Denki University, Tokyo, Japan}
\email{fujisawa@mail.dendai.ac.jp}
\dedicatory{Dedicated to the memory of Professor Sampei Usui}

\DeclareMathOperator{\Supp}{Supp}

\DeclareMathOperator{\Ass}{Ass}

\DeclareMathOperator{\Gr}{Gr}

\DeclareMathOperator{\image}{Image}
\DeclareMathOperator{\red}{red}
\DeclareMathOperator{\rank}{rank}
\DeclareMathOperator{\coker}{Coker}
\DeclareMathOperator{\ord}{ord}
\DeclareMathOperator{\id}{id}
\DeclareMathOperator{\kos}{Kos}
\DeclareMathOperator{\rec}{rec}
\DeclareMathOperator{\kernel}{Ker}
\DeclareMathOperator{\sym}{Sym}
\DeclareMathOperator{\sing}{Sing}
\DeclareMathOperator{\dlog}{dlog}

\newcommand{\gp}{^{\rm gp}}
\newcommand{\sm}{_{\rm sm}}

\newtheorem{thm}{Theorem}[section]
\newtheorem{lem}[thm]{Lemma}
\newtheorem{cor}[thm]{Corollary}

\theoremstyle{definition}
\newtheorem{defn}[thm]{Definition}
\newtheorem{rem}[thm]{Remark}
\newtheorem*{ack}{Acknowledgments}  
\newtheorem{step}{Step}
\newtheorem{say}[thm]{}
\makeatletter

\@addtoreset{equation}{section}
\makeatother
\begin{document}

\begin{abstract} 
We discuss variations of mixed Hodge 
structure arising from projective morphisms of complex analytic spaces. 
Then we treat generalizations of Koll\'ar's 
torsion-free theorem, vanishing theorem, and so on, 
for reducible complex analytic spaces as an application. 
The results will play a crucial role in the theory of minimal models 
for projective morphisms between complex analytic spaces. 
\end{abstract}

\maketitle 

\tableofcontents

\section{Introduction}\label{z-sec1}

This paper is a part of the first author's project 
to establish a minimal model theory for projective 
morphisms of complex analytic spaces. We have already 
known that the theory of mixed Hodge structures on cohomology 
with compact support is very useful in the theory of 
minimal models for higher-dimensional {\em{algebraic}} 
varieties (see \cite{fujino-fundamental}, 
\cite[Chapters 5 and 6]{fujino-foundations}, 
\cite{fujino-on-quasi}, and so on). Recently, the 
first author generalized the framework of the minimal 
model program for projective morphisms between 
complex analytic spaces (see \cite{fujino-minimal}, 
\cite{fujino-analytic-vanishing}, 
\cite{fujino-cone-contraction}, 
\cite{fujino-quasi-log-analytic}, and so on). 
From the Hodge theoretic point of view, the traditional 
framework of the minimal model theory for kawamata 
log terminal pairs is {\em{pure}}. For varieties and 
pairs whose singularities are worse than kawamata log 
terminal, it is natural to use the theory of {\em{mixed}} 
Hodge structures. Note that we cannot directly apply the 
Kawamata--Viehweg vanishing theorem or the original Fujita--Zucker--Kawamata 
semipositivity theorem to log canonical pairs 
and semi-log canonical pairs. Therefore, we 
want to obtain more general vanishing theorems, 
semipositivity theorems, and so on, in a suitable complex 
analytic setting by using the 
theory of mixed Hodge structures. 
In this paper, we will prove the 
following theorem, which is an analytic generalization of 
\cite[Theorems 7.1 and 7.3]{fujino-fujisawa}. 
Note that $f\colon (X, D)\to Y$ is assumed to be 
{\em{algebraic}} in \cite{fujino-fujisawa}. 
Our approach 
in this paper is slightly 
different from the one in \cite{fujino-fujisawa} (see 
Remark \ref{z-rem1.6} below). 
This paper is also a supplement to \cite{fujino-analytic-vanishing}. 
Although \cite{fujino-analytic-vanishing} depends on 
Saito's theory 
of mixed Hodge modules (see \cite{saito1}, 
\cite{saito2}, \cite{saito3}, 
\cite{saito4}, \cite{fujino-fujisawa-saito}, and 
\cite{saito5}), we do not use it here. 
We will use only variations of mixed Hodge structure. 

\begin{thm}[{Canonical extensions of Hodge 
bundles, see \cite[Theorems 7.1 
and 7.3]{fujino-fujisawa}}]\label{z-thm1.1}
Let $(X, D)$ be an analytic simple normal crossing pair 
such that $D$ is reduced and let $f\colon X\to Y$ be a 
proper surjective morphism onto a smooth complex variety $Y$. 
Assume that every stratum of $(X, D)$ is
dominant onto $Y$. 
Let $\Sigma$ be a normal crossing divisor 
on $Y$ such that 
every stratum of $(X, D)$ is smooth over $Y^*:=Y\setminus 
\Sigma$. We put $X^*:=f^{-1}(Y^*)$, $D^*:=D|_{X^*}$, 
and $d:=\dim X-\dim Y$. 
If we further assume that every stratum of $(X, D)$
is a K\"ahler manifold, 
then we have: 
\begin{itemize}
\item[\em{(i)}]
$R^k(f|_{X^*\setminus D^*})_!\mathbb R_{X^* \setminus D^*}$
underlies a graded polarizable 
variation of $\mathbb R$-mixed Hodge structure on $Y^*$ for 
every $k$. 
\end{itemize} 
We put 
\begin{equation*}
\mathcal V^k_{Y^*}:=R^k(f|_{X^*\setminus D^*})_!\mathbb R_{X^*\setminus 
D^*}\otimes \mathcal O_{Y^*}
\end{equation*} 
for every $k$. The Hodge filtration
and the weight filtration
on $\mathcal V^k_{Y^*}$
are denoted by $F$ and $L$ respectively. 
Moreover, the lower canonical extension
of $\mathcal{V}^k_{Y^*}$
is denoted by ${}^l\mathcal{V}^k_{Y^*}$.
The weight filtration $L$ on $\mathcal V^k_{Y^*}$ is extended
to ${}^l\mathcal{V}^k_{Y^*}$ by
$L_m({}^l\mathcal{V}^k_{Y^*})={}^lL_m(\mathcal{V}^k_{Y^*})$
for every $m$. 
Then we have the following: 
\begin{itemize}
\item[\em{(ii)}] 
There exists a unique finite decreasing filtration $F$
on ${}^l\mathcal{V}^k_{Y^*}$ such that
\begin{itemize}
\item[$\bullet$]
$F^p({}^l\mathcal{V}^k_{Y^*})|_{Y^*} \simeq F^p(\mathcal{V}^k_{Y^*})$, and
\item[$\bullet$]
$\Gr_F^p\Gr_m^L({}^l\mathcal{V}^k_{Y^*})$
is a locally free $\mathcal{O}_Y$-module of finite rank
\end{itemize}
for every $k,m,p$. 
\item[\em{(iii)}]
$R^{d-i}f_*\mathcal O_X(-D)$
is isomorphic to 
\begin{equation*}
\Gr^0_F ({}^l\mathcal V ^{d-i}_{Y^*})
=F^0({}^l\mathcal V ^{d-i}_{Y^*}) 
/F^1({}^l\mathcal V ^{d-i}_{Y^*}) 
\end{equation*} 
for every $i$.
In particular, $R^{d-i}f_*\mathcal O_X(-D)$ is locally free 
for every $i$. 
\item[\em{(iv)}]
$R^if_*\omega_{X/Y}(D)$
is isomorphic to 
\begin{equation*}
\left(\Gr^0_F({}^l\mathcal{V}^{d-i}_{Y^*})\right)^* 
=\mathcal H om_{\mathcal O_Y}(\Gr^0_F 
({}^l\mathcal V^{d-i}_{Y^*}), \mathcal O_Y) 
\end{equation*} 
for every $i$. 
In particular, $R^if_*\omega_{X/Y}(D)$ is locally free 
for every $i$. 
\end{itemize}
\end{thm}
For the precise definition of the upper and the lower canonical 
extensions in Theorem \ref{z-thm1.1},
see \ref{def-added} below.

To the best knowledge of the authors, 
Theorem \ref{z-thm1.1} is new even when $X$ is 
irreducible. For various geometric applications, 
it seems to be indispensable to treat the case where $X$ is reducible. 
When $X$ and $Y$ are algebraic and $f\colon X\to Y$ is projective, 
Theorem \ref{z-thm1.1} (see \cite[Theorems 7.1 and 7.3]{fujino-fujisawa}) 
is one of the main ingredients of \cite{fujino-moduli}, 
where we have to allow $X$ reducible. 

\begin{rem}
\label{z-rem1.2}
We do not need the relative monodromy weight filtration for applications 
in the theory of minimal models (see, for example, 
\cite{fujino-cone-contraction}). 
We are mainly interested in Hodge bundles and their extensions 
(see also \cite{fujino-fujisawa3}). 
However, we can prove the existence of the relative 
monodromy weight filtration in Theorem \ref{z-thm1.1}. 
In fact,
the variations of $\mathbb{R}$-mixed Hodge structure
in (i) of the theorem above are admissible.
This can be checked by the same argument as in \cite{fujino-fujisawa}. 
For the details, see Remark \ref{rem:2} below.
\end{rem}

By Theorem 
\ref{z-thm1.1}, we can use a generalization of the Fujita--Zucker--Kawamata 
semipositivity theorem in the complex 
analytic setting. 

\begin{thm}[Semipositivity]\label{z-thm1.3} 
In Theorem \ref{z-thm1.1}, 
we further assume that 
every local monodromy on the local system 
$R^{d-i}(f|_{X^*\setminus D^*})_!\mathbb R_{X^*\setminus D^*}$
around $\Sigma$ is unipotent. 
Let $\varphi\colon V\to X$ be any morphism 
from a projective variety $V$. 
Then $\varphi^*R^if_*\omega_{X/Y}(D)$ is a nef 
locally free sheaf on $V$.  
\end{thm}

In order to prove Theorem \ref{z-thm1.1}, 
we will establish: 

\begin{thm}[Weight spectral sequence]\label{z-thm1.4}
Let $(X, D)$ be an analytic simple normal crossing pair 
such that $D$ is reduced and 
let $f\colon X\to Y$ be a proper 
morphism between complex analytic spaces. 
We assume that $Y$ is a smooth complex variety and that 
there exists a normal crossing divisor $\Sigma$ on $Y$ such that 
every stratum of 
$(X, D)$ is dominant onto $Y$, and smooth over $Y\setminus \Sigma$. 
If we assume that every stratum of $(X, D)$
is a K\"ahler manifold in addition, 
then we have a spectral sequence: 
\begin{equation*}
E^{p, q}_1=\bigoplus_S
R^qf_*\mathcal O_S\Rightarrow R^{p+q}f_*\mathcal O_X(-D), 
\end{equation*} 
where 
$S$ runs through all $(\dim X-p)$-dimensional 
strata of $(X, D)$, 
such that it degenerates at $E_2$ and 
its $E_1$-differential $d_1$ splits.
Moreover,
$R^if_*\mathcal O_X(-D)$
is locally free of finite rank for every $i$.
\end{thm}

By combining Theorem \ref{z-thm1.4} with 
Takegoshi's results (see \cite{takegoshi}), we can prove: 

\begin{thm}[Torsion-freeness and vanishing theorem]\label{z-thm1.5}
Let $(X, D)$ be an analytic simple normal crossing pair 
such that $D$ is reduced and 
let $f\colon X\to Y$ be a projective 
morphism between complex analytic spaces. 
We assume that $Y$ is a complex variety and that every stratum of 
$(X, D)$ is dominant onto $Y$. 
Then we have the following properties. 
\begin{itemize}
\item[\em{(i)}] $($Torsion-freeness$)$.~$R^qf_*\omega_X(D)$ 
is a torsion-free sheaf for every $q$. 
\item[\em{(ii)}] $($Vanishing theorem$)$.~Let $\pi\colon Y\to Z$ be 
a projective morphism between complex analytic 
spaces and let $\mathcal A$ be a $\pi$-ample 
line bundle on $Y$. 
Then 
\begin{equation*}
R^p\pi_*\left(\mathcal A\otimes R^qf_*\omega_X(D)\right)=0
\end{equation*} 
holds  
for every $p>0$ and every $q$. 
\end{itemize}
\end{thm}

Of course, Theorem \ref{z-thm1.5} is a generalization of 
Koll\'ar's torsion-freeness and vanishing theorem (see 
\cite{kollar1}) for reducible 
complex analytic spaces. 
We make a remark on the relationship between 
\cite{fujino-fujisawa} and this paper. 

\begin{rem}\label{z-rem1.6} 
In \cite{fujino-fujisawa}, 
we have already treated Theorems \ref{z-thm1.1} and 
\ref{z-thm1.5} when $X$ and $Y$ are algebraic 
and $f\colon X\to Y$ is projective. 
Roughly speaking, in \cite[\S 6]{fujino-fujisawa}, 
we first establish Theorem \ref{z-thm1.5} 
when $X$ is quasi-projective and $f\colon X\to Y$ is algebraic. 
Then, by using it, we prove Theorem \ref{z-thm1.1} 
under the assumption that $X$ and $Y$ are algebraic 
and $f\colon X\to Y$ is projective 
in \cite[\S 7]{fujino-fujisawa}. 
When $X$ is quasi-projective, we can use the theory of 
mixed Hodge structures. 
Hence we can obtain desired vanishing theorems and 
torsion-freeness without using the theory of variations of 
mixed Hodge structure 
(for the details, see \cite[Chapter 5]{fujino-foundations}). 
In this paper, we will directly prove Theorems \ref{z-thm1.1} 
and \ref{z-thm1.4} with the aid of some results established 
for K\"ahler manifolds 
(see \cite{takegoshi}). 
Then, we will prove Theorem \ref{z-thm1.5} as an application. 
Theorem \ref{z-thm1.4} is new even when $X$ and $Y$ are algebraic 
and $f\colon X\to Y$ is projective. 
\end{rem}

By using Theorem \ref{z-thm1.5}, we have: 

\begin{thm}[{see \cite[Theorem 3.1]{fujino-analytic-vanishing}}]
\label{z-thm1.7}
Let $(X, D)$ be an analytic simple normal crossing pair 
such that $D$ is reduced and 
let $f\colon X\to Y$ be a projective 
morphism between complex analytic spaces. 
Then we have the following properties. 
\begin{itemize}
\item[\em{(i)}] $($Strict support condition$)$.~Every 
associated subvariety of $R^qf_*\omega_X(D)$ 
is the $f$-image of some stratum of $(X, D)$ for every $q$. 
\item[\em{(ii)}] $($Vanishing theorem$)$.~Let $\pi\colon Y\to Z$ be 
a projective morphism between complex analytic 
spaces and let $\mathcal A$ be a $\pi$-ample 
line bundle on $Y$. 
Then 
\begin{equation*}
R^p\pi_*\left(\mathcal A\otimes R^qf_*\omega_X(D)\right)=0
\end{equation*} 
holds  
for every $p>0$ and every $q$. 
\item[\em{(iii)}] $($Injectivity theorem$)$.~Let $\mathcal L$ be an $f$-semiample 
line bundle on $X$. Let $s$ be a nonzero element of $H^0(X, 
\mathcal L^{\otimes k})$ for some nonnegative integer $k$ such that 
the zero locus of $s$ does not contain any strata of $(X, D)$. 
Then, for every $q$, 
the map 
\begin{equation*}
\times s\colon R^qf_*\left(\omega_X(D)
\otimes \mathcal L^{\otimes l}\right)
\to R^qf_*\left(\omega_X(D)\otimes \mathcal L^{\otimes k+l}\right)
\end{equation*} 
induced by $\otimes s$ is injective for every positive integer $l$. 
\end{itemize}
\end{thm}

Note that Theorem \ref{z-thm1.7} was first 
obtained in \cite[Theorem 3.1]{fujino-analytic-vanishing} 
under a weaker assumption that $f\colon X\to Y$ is 
K\"ahler by using Saito's theory of mixed Hodge modules. 
Theorems \ref{z-thm1.8} and \ref{z-thm1.9} are the 
main results of \cite{fujino-analytic-vanishing}. 
Although they may look artificial and technical, they are very 
useful and indispensable for the study of 
varieties and pairs whose singularities are 
worse than kawamata log terminal (see 
\cite{ambro}, \cite[Chapter 6]{fujino-foundations}, 
\cite{fujino-hyperbolic}, \cite{fujino-on-quasi}, 
\cite{fujino-cone-contraction}, \cite{fujino-quasi-log-analytic}, and so on).  
In \cite{fujino-analytic-vanishing}, we showed that 
Theorems \ref{z-thm1.8} and \ref{z-thm1.9} follow 
from Theorem \ref{z-thm1.7} (i) and (ii). 
Note that Theorem \ref{z-thm1.7} (iii) is an easy 
consequence of Theorem \ref{z-thm1.7} (i) and (ii). 
Hence this paper gives an approach to Theorems 
\ref{z-thm1.8} and \ref{z-thm1.9} without using 
Saito's theory of mixed Hodge modules. 

\begin{thm}[{see \cite[Theorem 1.1]{fujino-analytic-vanishing}}]\label{z-thm1.8}
Let $(X, \Delta)$ be an analytic simple 
normal crossing pair such that $\Delta$ is a boundary 
$\mathbb R$-divisor on $X$. 
Let $f\colon X\to Y$ be a projective morphism 
to a complex analytic space $Y$ and let $\mathcal L$ 
be a line bundle on $X$. 
Let $q$ be an arbitrary nonnegative integer. 
Then we have the following properties. 
\begin{itemize}
\item[\em{(i)}] $($Strict support condition$)$.~If 
$\mathcal L-(\omega_X+\Delta)$ is $f$-semiample,  
then every 
associated subvariety of $R^qf_*\mathcal L$ is the $f$-image 
of some stratum of $(X, \Delta)$. 
\item[\em{(ii)}] $($Vanishing theorem$)$.~If 
$\mathcal L-(\omega_X+\Delta)\sim _{\mathbb R} f^*\mathcal H
$ holds 
for some $\pi$-ample 
$\mathbb R$-line bundle $\mathcal H$ on $Y$, where 
$\pi\colon Y\to Z$ is a 
projective morphism to a complex analytic space 
$Z$, then we have 
$
R^p\pi_*R^qf_*\mathcal L=0
$
for every $p>0$. 
\end{itemize} 
\end{thm}

\begin{thm}[{Vanishing theorem of Reid--Fukuda type, 
see \cite[Theorem 1.2]{fujino-analytic-vanishing}}]\label{z-thm1.9}
Let $(X, \Delta)$ be an analytic simple 
normal crossing pair such that $\Delta$ is a boundary 
$\mathbb R$-divisor on $X$. 
Let $f\colon X\to Y$ and $\pi\colon Y\to Z$ be projective morphisms 
between complex analytic spaces and let $\mathcal L$ 
be a line bundle on $X$. 
If $\mathcal L-(\omega_X+\Delta)\sim _{\mathbb R} f^*\mathcal H$ 
holds such that $\mathcal H$ is an $\mathbb R$-line bundle, 
which is nef and 
log big over $Z$ with respect to 
$f\colon (X, \Delta)\to Y$, on $Y$, then  
$R^p\pi_*R^qf_*\mathcal L=0$ holds for every $p>0$ and every $q$. 
\end{thm}

In this paper, we do not prove Theorems \ref{z-thm1.8} and 
\ref{z-thm1.9}. 
For the details of Theorems \ref{z-thm1.8} and \ref{z-thm1.9}, 
see \cite{fujino-analytic-vanishing}. 
We note that we can find a completely different approach to Theorems 
\ref{z-thm1.8} and \ref{z-thm1.9} in \cite{murayama}. 
Although the motivation of the first author 
is obviously  the minimal model theory for projective 
morphisms between complex analytic spaces, 
we do not treat the minimal model program in this paper. 
We recommend that the interested reader looks at 
\cite{fujino-minimal}, \cite{fujino-cone-contraction}, 
\cite{fujino-quasi-log-analytic}, and so on. 
Theorems \ref{z-thm1.8} and \ref{z-thm1.9} have already 
played a crucial role in \cite{fujino-cone-contraction} and 
\cite{fujino-quasi-log-analytic}, where 
we established the fundamental theorems 
of the theory of minimal models for 
projective morphisms between complex analytic spaces. 
Anyway, by this paper, \cite{fujino-cone-contraction} 
and \cite{fujino-quasi-log-analytic} become free from 
Saito's theory of mixed Hodge modules. 
The relationship between \cite{fujino-analytic-vanishing} 
and this paper is as follows. 

\begin{rem}\label{z-rem1.10} 
In \cite[Corollary 1 and 4.7.~Remark]{fujino-fujisawa-saito} (see 
\cite[Theorem 2.6]{fujino-analytic-vanishing}), we constructed 
a weight spectral sequence of 
mixed Hodge modules. 
It is much more general than Theorem \ref{z-thm1.4} in some 
sense. 
Roughly speaking, 
\cite[Theorem 2.6]{fujino-analytic-vanishing} 
treats a dual statement of Theorem \ref{z-thm1.4} 
in a more general setting using Saito's theory of mixed Hodge modules. 
By combining it with Takegoshi's 
results (see \cite{takegoshi}), 
we proved Theorems \ref{z-thm1.7}, \ref{z-thm1.8}, 
\ref{z-thm1.9}, and so on, in \cite{fujino-analytic-vanishing}. 
From the Hodge theoretic viewpoint, one of the 
main ingredients of this paper is Steenbrink's result 
obtained in \cite{steenbrink1} and \cite{steenbrink2}. 
In our approach in this paper, 
Theorem \ref{z-thm1.4} is natural and sufficient for our purposes. 
\end{rem}

We look at the organization of this paper. 
In Section \ref{z-sec2}, we will briefly explain basic 
definitions and results necessary for this paper. 
In Subsection \ref{z-subsec2.1}, we will explain some useful lemmas 
on analytic simple normal crossing pairs. In 
Subsection \ref{z-subsec2.2}, we will briefly review 
Koll\'ar's package in the complex analytic setting. 
In Section \ref{z-sec3}, we will 
present abstract arguments
which help us to obtain a variation of mixed Hodge structure
whose Hodge filtration can be extended
to its canonical extension along a simple normal crossing divisor.
Section \ref{c-sec3} is the main part of this paper, 
where we will prove Theorems \ref{z-thm1.1} and 
\ref{z-thm1.4}. 
We will also see that a generalization of the 
Fujita--Zucker--Kawamata semipositivity theorem 
holds in the complex analytic setting (see Theorem \ref{z-thm1.3}). 
In Section \ref{z-sec5}, we will prove Theorem 
\ref{z-thm1.5}. In Section \ref{z-sec6}, 
we will prove Theorem \ref{z-thm1.7}. 
Section \ref{c-sec6} is a supplementary section, where 
we will explain a new construction of the rational 
structure for the cohomological $\mathbb Q$-mixed 
Hodge complex in \cite{steenbrink2}. 
We hope that it will help the reader to understand 
\cite{steenbrink1} and \cite{steenbrink2}. 

In \cite{fujino-fujisawa3}, 
we will use the same approach to variations of mixed 
Hodge structure as in Sections \ref{z-sec3}, 
\ref{c-sec3}, and \ref{c-sec6}. 

\begin{ack}\label{z-ack}
The authors thank Yuta Kusakabe very much for answering their questions.
The first author was partially 
supported by JSPS KAKENHI Grant Numbers 
JP19H01787, JP20H00111, JP21H00974, JP21H04994. 
The second author was partially
supported by JSPS KAKENHI Grant Number 
JP20K03542. 
The authors would like to thank the referee very much 
for useful comments. 
\end{ack}

In this paper, every complex analytic space is assumed to be 
{\em{Hausdorff}} and {\em{second-countable}}. 
Note that an irreducible and reduced 
complex analytic space is called a {\em{complex variety}}.
We will freely use the basic results on complex analytic 
geometry in \cite{banica} and \cite{fischer}. 

\section{Preliminaries}\label{z-sec2}

In this section, we will collect some basic definitions. 
Let us start with the definition of {\em{analytic 
simple normal crossing pairs}}. 

\begin{defn}[Analytic simple normal crossing pairs]\label{z-def2.1}
Let $X$ be a simple normal crossing divisor 
on a smooth complex analytic space $M$ and 
let $B$ be an $\mathbb R$-divisor on $M$ such that 
the support of $B+X$ is a simple normal crossing divisor on $M$ and 
that $B$ and $X$ have no common irreducible components. 
Then we put $D:=B|_X$ and 
consider the pair $(X, D)$. 
We call $(X, D)$ an {\em{analytic globally embedded simple 
normal crossing pair}} and $M$ the {\em{ambient space}} 
of $(X, D)$. 
If the pair $(X, D)$ is locally isomorphic to an analytic 
globally embedded 
simple normal crossing pair at any point of $X$ and the irreducible 
components of $X$ and $D$ are all smooth, 
then $(X, D)$ is called an {\em{analytic simple normal crossing 
pair}}. 

When $(X, D)$ is an analytic simple normal crossing 
pair, $X$ has an invertible 
dualizing sheaf $\omega_X$. 
We usually use the symbol $K_X$ as a formal 
divisor class with an isomorphism 
$\mathcal O_X(K_X)\simeq \omega_X$ if there is 
no danger of confusion. 
We note that we cannot always define $K_X$ globally 
with $\mathcal O_X(K_X)\simeq \omega_X$. 
In general, it only exists locally on $X$. 
\end{defn}

The notion of log canonical centers is not so easy to grasp. 
In this paper, the following basic properties, 
which follow easily from the definition, 
are sufficient to understand for our purposes. 

\begin{lem}\label{z-lem2.2} 
Let $V$ be a smooth complex variety. 
\begin{itemize}
\item[{\em{(i)}}] Let $F$ be 
a simple normal crossing divisor on $V$. 
Let $F=\sum _{i\in I} F_i$ be the irreducible 
decomposition. Then $C$ is a log canonical 
center of $(V, F)$ if and only if $C$ is an irreducible 
component of $F_{i_1}\cap \cdots \cap F_{i_k}$ 
for some nonempty subset $\{i_1, \ldots, i_k\}$ of $I$. 
\item[{\em{(ii)}}] Let $G$ be 
an $\mathbb R$-divisor on $V$ such that $\Supp G$ 
is a simple normal crossing divisor. We put 
$G=\sum _i a_i G_i$, where $a_i\in \mathbb R$ 
and $G_i$ is a prime divisor on $V$ for every 
$i$ such that $G_i\ne G_j$ for $i\ne j$. 
Then $C$ is a log canonical center of 
$(V, G)$ if and only if $C$ is a log 
canonical center of $(V, \sum _{a_i=1} G_i)$. 
\end{itemize}
\end{lem}

For the exact definition and basic 
properties of log canonical centers, 
see, for example, \cite[Section 2.3]{fujino-foundations}, 
\cite[Section 2.1]{fujino-cone-contraction}, and so on. 
The notion of {\em{strata}} plays a crucial role in this paper. 

\begin{defn}[Strata]\label{z-def2.3}
Let $(X, D)$ be an analytic simple normal crossing pair 
as in Definition \ref{z-def2.1}. 
Let $\nu\colon X^\nu\to X$ be the normalization. 
We put 
\begin{equation*}
K_{X^\nu}+\Theta=\nu^*(K_X+D). 
\end{equation*} 
This means that $\Theta$ is the union of $\nu^{-1}_*D$ and the 
inverse image of the singular locus of $X$. 
We note that $X^\nu$ is smooth and the support of $\Theta$ 
is a simple normal crossing divisor on $X^\nu$. 
If $W$ is an irreducible component of $X$ 
or the $\nu$-image 
of some log canonical center of $(X^\nu, \Theta)$, 
then $W$ is called a {\em{stratum}} of $(X, D)$. 
\end{defn}

We make an important remark. 

\begin{rem}\label{z-rem2.4}
In this paper, $D$ is always assumed to be reduced. 
Hence, $\Theta$ in Definition \ref{z-def2.3} is a 
reduced simple normal crossing divisor on $X^\nu$. 
We do not need $\mathbb Q$-divisors nor 
$\mathbb R$-divisors in this paper. 
This means that it is sufficient 
to treat the case where $B$ in Definition \ref{z-def2.1} 
is a $\mathbb Z$-divisor. 
\end{rem}

We recall Siu's theorem on complex analytic sheaves, 
which is a special case of \cite[Theorem 4]{siu}. 
We need it for Theorem \ref{z-thm1.7} (i) and 
Theorem \ref{z-thm1.8} (i). 

\begin{thm}\label{z-thm2.5} 
Let $\mathcal F$ be a coherent sheaf on a complex 
analytic space $X$. 
Then there exists a locally finite family $\{Y_i\}_{i\in I}$ 
of complex analytic subvarieties of $X$ such that 
\begin{equation*}
\Ass _{\mathcal O_{X,x}}(\mathcal F_x)=\{\mathfrak{p}_{x, 1}, 
\ldots, \mathfrak{p}_{x, r(x)}\}
\end{equation*}
holds for every point $x\in X$, where 
$\mathfrak{p}_{x, 1}, 
\ldots, \mathfrak{p}_{x, r(x)}$ are the prime ideals 
of $\mathcal O_{X, x}$ associated to the irreducible components
of the germs $Y_{i, x}$ of $Y_i$ at $x$ with $x\in Y_i$. 
We note that each $Y_i$ is called an {\em{associated subvariety}} 
of $\mathcal F$. 
\end{thm}

\begin{defn}[Relatively nef, ample, and big line bundles]\label{z-def2.6}
Let $f\colon X\to Y$ be a projective 
morphism of complex analytic spaces and let 
$\mathcal L$ be a line bundle on $X$. 
Then we say that 
\begin{itemize}
\item $\mathcal L$ is {\em{$f$-nef}} if $\mathcal L\cdot 
C\geq 0$ holds for every curve $C$ on $X$ such that 
$f(C)$ is a point, and 
\item $\mathcal L$ is {\em{$f$-ample}} if 
$\mathcal L|_{f^{-1}(y)}$ is ample 
in the usual sense for every $y\in Y$. 
\end{itemize} 
We further assume that $f\colon X\to Y$ is 
a projective surjective morphism of complex varieties. 
Then we say that 
\begin{itemize}
\item $\mathcal L$ is {\em{$f$-big}} 
if there exists some positive real number 
$c$ such that $\rank f_*\mathcal L^{\otimes m}> c\cdot m^d$ holds 
for $m \gg 0$, where $d=\dim X-\dim Y$. 
\end{itemize}
\end{defn}

We need the notion of {\em{nef locally free sheaves}} in 
Theorem \ref{z-thm1.3}. 

\begin{defn}[Nef locally free sheaves]\label{z-def2.7} 
Let $\mathcal E$ be a locally free sheaf of finite rank on a projective 
variety $V$. 
If $\mathcal O_{\mathbb P_V(\mathcal E)}(1)$ is nef, 
that is, $\mathcal O_{\mathbb P_V(\mathcal E)}(1)\cdot 
C\geq 0$ holds for every curve $C$ on $\mathbb P_V(\mathcal E)$, 
then $\mathcal E$ is called a {\em{nef}} 
locally free sheaf on $V$. 
\end{defn}

A nef locally free sheaf is sometimes called a 
{\em{semipositive vector bundle}} or a {\em{semipositive 
locally free sheaf}} in the literature. 

\subsection{Lemmas on analytic simple normal crossing pairs}
\label{z-subsec2.1}

In this subsection, we will collect some useful lemmas 
on analytic simple normal crossing pairs. 
We will repeatedly use these lemmas in subsequent sections. 

\begin{lem}[{see \cite[Lemmas 2.13 
and 2.15]{fujino-analytic-vanishing}}]\label{z-lem2.8}
Let $(X, D)$ and $(X', D')$ be simple normal 
crossing pairs such that $D$ and $D'$ are reduced. 
Let $g\colon X'\to X$ be a projective bimeromorphic 
morphism. 
Assume that there exists a Zariski open subset $U$ of 
$X$ such that $g\colon U':=g^{-1}(U)\to U$ is an isomorphism 
and that $U$ 
{\em{(}}resp.~$U'${\em{)}} intersects every stratum of 
$(X, D)$ {\em{(}}resp.~$(X', D')${\em{)}}. 
Then 
$
R^ig_*\mathcal O_{X'}=0 
$ and 
$
R^ig_*\mathcal O_{X'}(K_{X'}+D')=0
$
for every $i>0$, and 
$
g_*\mathcal O_{X'}\simeq \mathcal O_X 
$ and 
$
g_*\mathcal O_{X'}(K_{X'}+D')\simeq \mathcal O_X(K_X+D)
$ hold. 
\end{lem}

\begin{proof}
By \cite[Lemma 2.15]{fujino-analytic-vanishing}, 
we have $R^ig_*\mathcal O_{X'}=0$ for every 
$i>0$ and $g_*\mathcal O_{X'}\simeq \mathcal O_X$. 
Since $D$ and $D'$ are reduced, 
we can easily check that 
\begin{equation}\label{z-eq2.1}
K_{X'}+D'=g^*(K_X+D)+E
\end{equation} 
holds for some effective $g$-exceptional 
Cartier divisor $E$ on $X'$ and that $D'=g^{-1}_*D$ holds. 
By \eqref{z-eq2.1}, we have 
$g_*\mathcal O_{X'}(K_{X'}+D')\simeq 
\mathcal O_X(K_X+D)$. 
By \cite[Lemma 2.13]{fujino-analytic-vanishing}, 
we obtain $R^ig_*\mathcal O_{X'}(K_{X'}+D')=0$ for 
every $i>0$. 
We finish the proof. 
\end{proof}

\begin{lem}[{see \cite[Lemma 5.1]{fujino-analytic-vanishing}}]\label{z-lem2.9}
Let $(X, D)$ be an analytic 
simple normal crossing pair such 
that $D$ is reduced and 
let $f\colon X\to Y$ be a projective 
morphism between complex analytic spaces. 
Let $L$ be a Cartier divisor on $X$. 
We take an arbitrary point $P\in Y$. 
Then, after shrinking $Y$ around $P$ suitably, 
we can construct the following commutative 
diagram: 
\begin{equation*}
\xymatrix{
Z\ar[d]_-p\ar@{^{(}->}[r]^-\iota& M\ar[dd]^-q\\ 
X\ar[d]_f \\ 
Y\ar@{^{(}->}[r]_-{\iota_Y}& \Delta^m
}
\end{equation*}
such that 
\begin{itemize}
\item[\em{(i)}] $\iota_Y\colon Y\hookrightarrow \Delta^m$ is a 
closed embedding into a polydisc $\Delta^m$ with $\iota_Y(P)=0\in 
\Delta^m$, 
\item[\em{(ii)}] $(Z, D_Z)$ is an analytic globally embedded simple 
normal crossing pair such that $D_Z$ is reduced, 
\item[\em{(iii)}] $M$ is the ambient space of $(Z, D_Z)$ and 
is projective over $\Delta^m$, 
\item[\em{(iv)}] there exists a Cartier divisor $L_Z$ on $Z$ 
satisfying 
\begin{equation*}
L_Z-(K_Z+D_Z)=p^*(L-(K_X+D)), 
\end{equation*} 
$p_*\mathcal O_Z(L_Z)\simeq \mathcal O_X(L)$, and 
$R^ip_*\mathcal O_Z(L_Z)=0$ for 
every $i>0$, 
\item[\em{(v)}] $p(W)$ is a stratum of $(X, D)$ for every 
stratum $W$ of $(Z, D_Z)$, 
\item[\em{(vi)}] there exists a Zariski open subset $U$ of $X$, which 
intersects every stratum of $X$, such that 
$p$ is an isomorphism over $U$,  
\item[\em{(vii)}] $p$ maps every stratum of $Z$ 
bimeromorphically onto some stratum of $X$, and 
\item[\em{(viii)}] for any stratum $S$ of $(X, D)$, there 
exists a stratum $W$ of $(Z, D_Z)$ such that $S=p(W)$.   
\end{itemize} 
In particular, we have: 
\begin{itemize}
\item[{\em{(ix)}}] $p_*\mathcal O_Z(K_Z+D_Z)\simeq 
\mathcal O_X(K_X+D)$ and $R^ip_*\mathcal O_Z(K_Z+D_Z)=0$ 
holds for every $i>0$. 
\end{itemize}
\end{lem}

\begin{proof}
The proof of \cite[Lemma 5.1]{fujino-analytic-vanishing}, 
where we allow $D$ to be a boundary $\mathbb R$-divisor, 
works without any modifications. 
Thus we have the desired commutative diagram satisfying 
(i)--(viii). When $L=K_X+D$, 
we have $L_Z=K_Z+D_Z$ by (iv). Hence 
we obtain (ix). 
\end{proof}

\begin{lem}
[{see \cite[Lemma 2.17]{fujino-analytic-vanishing}}]\label{z-lem2.10}
Let $(X, D)$ be an analytic globally embedded simple 
normal crossing pair such that $D$ is reduced 
and let $M$ be the ambient space 
of $(X, D)$. 
Let $C$ be a stratum of $(X, D)$, which 
is not an irreducible component of $X$. 
Let $\sigma\colon M'\to M$ be the blow-up along $C$ and let 
$X'$ denote the reduced structure of the total transform of 
$X$ on $M'$. 
We put 
\begin{equation*}
K_{X'}+D':=g^*(K_X+D), 
\end{equation*} 
where $g:=\sigma|_{X'}$. Then we have the following properties:
\begin{itemize}
\item[\em{(i)}] $(X', D')$ is an analytic globally embedded simple 
normal crossing pair such that $D'$ is reduced, 
\item[\em{(ii)}] $M'$ is the ambient space of $(X', D')$, 
\item[\em{(iii)}] $g_*\mathcal O_{X'}\simeq \mathcal O_X$ holds 
and $R^ig_*\mathcal O_{X'}=0$ for every $i>0$, 
\item[\em{(iv)}] the strata of $(X, D)$ are exactly the images of 
the strata of $(X', D')$, 
and 
\item[\em{(v)}] $\sigma^{-1}(C)$ is a maximal 
{\em{(}}with respect to the inclusion{\em{)}} stratum of $(X', D')$,
that is, $\sigma^{-1}(C)$ is an irreducible component of $X'$.  
\end{itemize}
\end{lem}

\begin{proof}
The proof of \cite[Lemma 2.17]{fujino-analytic-vanishing}, 
where we allow $D$ to be a boundary $\mathbb R$-divisor, 
works without any modifications. 
\end{proof}

\subsection{Complex analytic generalization of Koll\'ar's 
package}\label{z-subsec2.2}

Here, let us briefly review Koll\'ar's package (see 
\cite{kollar1} and \cite{kollar2}) in the complex analytic setting. 
We recommend that the interested reader looks at 
\cite{fujino-trans2}, 
\cite[Chapter V.~3.7.~Theorem]{nakayama3}, and \cite{takegoshi}. 

Theorem \ref{z-thm2.11} is a variant of 
Takegoshi's vanishing theorem (see 
\cite[Theorem IV Relative vanishing 
Theorem]{takegoshi} and \cite[Corollary 1.5]{fujino-trans2}). 
We note that it is well known when $f\colon X\to Y$ and 
$\pi\colon Y\to Z$ are 
projective morphisms of 
algebraic varieties. 

\begin{thm}[Vanishing theorem]\label{z-thm2.11}
Let $f\colon X \to Y$ and $\pi \colon Y\to Z$ 
be projective surjective morphisms 
between complex varieties such that $X$ is smooth. 
Let $\mathcal M$ be 
a line bundle on $Y$. Assume that $\mathcal M$ is 
$\pi$-nef and $\pi$-big over $Z$. 
Then 
\begin{equation}\label{z-eq2.2}
R^p\pi_*\left(
\mathcal M\otimes R^qf_*\omega_X\right)=0
\end{equation} 
holds 
for every $p>0$ and every $q$. 
In particular, if further $\pi$ is bimeromorphic, 
then 
\begin{equation}\label{z-eq2.3}
R^p\pi_*R^qf_*\omega_X=0
\end{equation} 
holds for every $p>0$ and every $q$. 
\end{thm}

\begin{proof} 
The vanishing theorem \eqref{z-eq2.2} is more or less well known to 
the experts. For the details, see, for example, 
\cite[Corollary 1.5]{fujino-trans2}. 
Let us quickly explain how to use 
\cite[Corollary 1.5]{fujino-trans2}. We 
take an arbitrary point $P\in Z$. 
By shrinking $Z$ around $P$, we may assume 
that $X$ is K\"ahler since $\pi$ and $f$ are projective. 
We put $g:=\pi$, $\mathcal L:=f^*\mathcal M$, $D:=0$, 
and $E:=\mathcal O_X$. Note that the 
trivial line bundle is a Nakano semipositive vector bundle. 
Then we have $R^p\pi_*(\mathcal M\otimes R^qf_*\omega_X)=0$ 
for every $p>0$ and every $q$ by \cite[Corollary 1.5]{fujino-trans2}. 
Note that \eqref{z-eq2.3} is a special case of 
\eqref{z-eq2.2}. 
This is because the trivial line bundle on $Y$ is $\pi$-nef 
and $\pi$-big when $\pi$ is bimeromorphic. 
\end{proof}

Lemma \ref{z-lem2.12} is an easy consequence of Theorem 
\ref{z-thm2.11}. 

\begin{lem}\label{z-lem2.12}
Let $f_i\colon X_i\to Y$ be a projective 
surjective morphism of complex varieties such that 
$X_i$ is smooth for every $1\leq i\leq k$. 
Let $\pi\colon Y\to Z$ be a projective bimeromorphic 
morphism between complex varieties. 
We put 
\begin{equation*}
\mathcal F:=\bigoplus _{i=1}^k R^{q_i}f_{i*}\omega_{X_i}, 
\end{equation*} 
where $q_i$ is some nonnegative integer for every $i$. 
Let $\mathcal G$ be a coherent sheaf on $Y$. 
Assume that $\mathcal G$ is a direct summand of $\mathcal F$.
Then we have $R^p\pi_*\mathcal F=0$ and 
$R^p\pi_*\mathcal G=0$ for every $p>0$. 
In particular, $\pi_*\mathcal G$ is a direct summand of 
\begin{equation*}
\bigoplus _{i=1}^k R^{q_i}(\pi\circ f_i)_* \omega_{X_i}. 
\end{equation*} 
\end{lem}

\begin{proof}
By Lemma \ref{z-thm2.11}, 
$R^p\pi_*R^{q_i}f_{i*}\omega_{X_i}=0$ holds  
for every $p>0$. Hence we have $R^p\pi_*\mathcal F=0$ for every $p>0$. 
Since $\mathcal G$ is a direct summand of $\mathcal F$, 
we obtain $R^p\pi_*\mathcal G=0$ for every $p>0$. 
Since $\pi_*\mathcal G$ is a 
direct summand of $\pi_*\mathcal F$ and we have the following 
isomorphisms:  
\begin{equation*}
\pi_*\mathcal F=\bigoplus _{i=1}^k \pi_*R^{q_i}f_{i*} \omega_{X_i}
\simeq 
\bigoplus _{i=1}^k R^{q_i}(\pi\circ f_i)_* \omega_{X_i},  
\end{equation*} 
it follows that $\pi_*\mathcal G$ is a direct summand of 
$\bigoplus _{i=1}^kR^{q_i}(\pi\circ f_i)_*\omega_{X_i}$. 
We finish the proof.  
\end{proof}

We note that the trivial line bundle is a Nakano 
semipositive vector bundle. Hence, 
Theorem \ref{z-thm2.13} below is a special case of 
Takegoshi's torsion-freeness 
(see \cite[Theorem II Torsion freeness Theorem]{takegoshi} 
and \cite[Corollary 1.2]{fujino-trans2}). 
When $f\colon X\to Y$ is a projective surjective morphism 
between projective varieties, it is nothing but Koll\'ar's 
famous torsion-freeness 
(see \cite[Theorem 2.1 (i)]{kollar1}). 

\begin{thm}[Torsion-freeness]\label{z-thm2.13}
Let $f\colon X\to Y$ be a projective surjective morphism 
of complex varieties such that $X$ is smooth. 
Then $R^qf_*\omega_X$ is torsion-free for every $q$. 
\end{thm}

\begin{proof} 
Let us see how to use 
\cite[Corollary 1.2]{fujino-trans2}. 
We take an arbitrary point $P\in Y$. 
By shrinking $Y$ around $P$, we may assume that $X$ 
is K\"ahler since $f$ is projective. 
We put $E:=\mathcal O_X$ and $F:=\mathcal O_X$ 
with trivial metrics. Then $R^qf_*\omega_X$ is 
torsion-free for every $q$ by \cite[Corollary 1.2]{fujino-trans2}. 
This is what we wanted. 
\end{proof}

Before stating Theorem \ref{z-thm2.14},
we briefly recall the definition
of the upper and the lower canonical extensions.
See e.g. \cite[Chapter \RomII]{DeligneED},
\cite[Section 2]{kollar2},
\cite[Remark 7.4]{fujino-fujisawa}
for more details. 

\begin{say}[Upper and lower canonical extensions]
\label{def-added} 
Let $Y$ be a smooth complex variety,
$\Sigma$ a reduced normal crossing divisor on $Y$,
and $(\mathcal{V},\nabla)$
a locally free $\mathcal{O}_Y$-module of finite rank $\mathcal{V}$
on $Y \setminus \Sigma$
equipped with an integrable connection $\nabla$.
Then the upper (resp.~lower) canonical extension of $\mathcal{V}$
is a locally free $\mathcal{O}_Y$-module
${}^u\mathcal{V}$ (resp.~${}^l\mathcal{V}$) of finite rank on $Y$
equipped with a log integrable connection ${}^u\nabla$
(resp.~${}^l\nabla$)
satisfying the following:
\begin{itemize}
\item
${}^u\mathcal{V}|_{Y \setminus \Sigma} \simeq \mathcal{V}$
(resp.~${}^l\mathcal{V}|_{Y \setminus \Sigma} \simeq \mathcal{V}$),
\item
${}^u\nabla|_{Y \setminus \Sigma}=\nabla$
(resp.~${}^l\nabla|_{Y \setminus \Sigma}=\nabla$)
via the identification above, and
\item
all the eigenvalues of the residue of ${}^u\nabla$
(resp.~${}^l\nabla$) along every irreducible component of $\Sigma$
are contained in $(-1,0]$ (resp.~$[0,1)$).
\end{itemize}
It is well known that
such ${}^u\mathcal{V}$ (resp.~${}^l\mathcal{V}$)
exists uniquely up to the unique isomorphism.
When it is not necessary to specify the (log) integrable connections,
we only mention ${}^u\mathcal{V}$ (resp.~${}^l\mathcal{V}$)
as the upper (resp.~lower) canonical extension of $\mathcal{V}$.
For example, if
$\mathcal{V}=\mathcal{O}_{Y \setminus \Sigma} \otimes_{\mathbb{C}} \mathbb{V}$
with some $\mathbb{C}$-local system $\mathbb{V}$ on $Y \setminus \Sigma$,
then ${}^u\mathcal{V}$ (resp.~${}^l\mathcal{V}$)
is the upper (resp.~lower) canonical extension of $\mathcal{V}$
equipped with the canonical integrable connection $\nabla=d \otimes \id$. 

For a polarizable variation of $\mathbb{R}$-Hodge structure
$(\mathbb{V},F)$ on $Y \setminus \Sigma$,
we set
$\mathcal{V}=\mathcal{O}_{Y \setminus \Sigma} \otimes_{\mathbb{R}} \mathbb{V}$
and denote ${}^u\mathcal{V}$ (resp.~${}^l\mathcal{V}$)
the upper (resp.~lower) canonical extension
of $\mathcal{V}$ as above.
Then
\begin{equation}
j_*F^p\mathcal{V} \cap {}^u\mathcal{V} \quad
\text{(resp.~}
j_*F^p\mathcal{V} \cap {}^l\mathcal{V}
\text{)},
\end{equation}
where $j$ is the open immersion $Y \setminus \Sigma \hookrightarrow Y$,
gives us a finite decreasing filtration $F$
on ${}^u\mathcal{V}$ (resp.~${}^l\mathcal{V}$)
called the upper (resp.~lower) canonical extension
of the Hode filtration $F$.
The nilpotent orbit theorem implies that
$\Gr_F^p({}^u\mathcal{V})$
(resp.~$\Gr_F^p({}^l\mathcal{V})$)
is locally free for every $p$.
For the nilpotent orbit theorem,
see \cite[(4.12)]{schmid},
\cite[Theorem A]{Deng}.
\end{say}

When $f\colon X\to Y$ is algebraic, Theorem 
\ref{z-thm2.14} below was first obtained independently 
by 
Koll\'ar (see \cite[Theorem 2.6]{kollar2}) 
and 
Nakayama (see \cite[Theorem 1]{nakayama2}). 
When $f\colon X\to Y$ is a projective morphism of 
smooth complex varieties, it was obtained by Moriwaki 
(see \cite[Theorem (2.4)]{moriwaki}). 

\begin{thm}[{Hodge filtration, see \cite[Theorem V Local 
freeness Theorem (ii)]{takegoshi}
and \cite[Chapter V, 3.7.~Theorem (4)]{nakayama3}}]\label{z-thm2.14}
Let $f\colon X\to Y$ be a proper surjective morphism 
between smooth complex varieties and let $\Sigma$ be 
a normal crossing divisor on $Y$ such that 
$f$ is smooth over $Y^*:=Y\setminus \Sigma$. 
We assume that $X$ is a K\"ahler manifold. 
Then $R^qf_*\omega_{X/Y}$ is locally free 
and is characterized as the upper canonical 
extension of the corresponding 
bottom Hodge filtration on $Y^*$ for 
every $q$. 
\end{thm}

We make a remark on the proof of Theorem \ref{z-thm2.14}. 

\begin{rem}\label{z-rem2.15} 
One of the main ingredients of \cite{nakayama2} 
is Steenbrink's result established 
in \cite{steenbrink1} and \cite{steenbrink2} 
(see \cite[Theorem 3]{nakayama2}). 
Although it was explicitly stated only for 
projective morphisms, it also holds for 
proper morphisms from K\"ahler manifolds 
(see Remark \ref{c-rem3.4} below). 
Hence the argument in \cite{nakayama2} works for K\"ahler 
manifolds with the aid of \cite{takegoshi}. 
We recommend that the interested reader looks 
at \cite[Conjectures 7.2 and 7.3]{nakayama1} and 
\cite{nakayama2}. 
\end{rem}

\section{On variations of mixed Hodge structure}\label{z-sec3}

In this section,
we present abstract arguments
which help us to obtain a variation of mixed Hodge structure
whose Hodge filtration can be extended
to its canonical extension along a simple normal crossing divisor.
Although the arguments look rather technical,
they give us an appropriate viewpoint
for the proofs of Theorems \ref{z-thm1.1} and \ref{z-thm1.4}. 
As we treat proper,
but not necessarily projective,
morphisms of complex analytic spaces
with certain K\"ahler conditions
in Theorems \ref{z-thm1.1} and \ref{z-thm1.4},
we consider (graded) polarizable variations
of (mixed) Hodge structure
defined over $\mathbb{R}$
because polarizations might not be defined over $\mathbb{Q}$
but are defined over $\mathbb{R}$.

\begin{say}
\label{para:1}
Let $X$ be a complex manifold and
\begin{equation}
K=((K_{\mathbb{R}}, W), (K_{\mathcal{O}}, W, F), \alpha)
\end{equation}
be a triple consisting of
\begin{itemize}
\item
a bounded complex of $\mathbb{R}$-sheaves $K_{\mathbb{R}}$ on $X$
equipped with a finite increasing filtration $W$,
\item
a bounded complex of $\mathcal{O}_X$-modules $K_{\mathcal{O}}$ on $X$
equipped with a finite increasing filtration $W$
and a finite decreasing filtration $F$, and
\item
a morphism of complexes of $\mathbb{R}$-sheaves
$\alpha \colon K_{\mathbb{R}} \to K_{\mathcal{O}}$
preserving the filtration $W$.
\end{itemize}
Such a triple $K$ yields a pair of spectral sequences
\begin{equation}
E_r^{p,q}(K,W)
=(E_r^{p,q}(K_{\mathbb{R}},W),
(E_r^{p,q}(K_{\mathcal{O}},W), F), E_r^{p,q}(\alpha))
\end{equation}
where $F$ on $E_r^{p,q}(K_{\mathcal{O}},W)$
stands for the first direct filtration
(cf. \cite[(1.3.8)]{DeligneII}), that is,
\begin{equation}
F^aE_r^{p,q}(K_{\mathcal{O}},W)
=\image(E_r^{p,q}(F^aK_{\mathcal{O}},W) \to E_r^{p,q}(K_{\mathcal{O}},W))
\end{equation}
for $r=0,1,2, \dots, \infty$ and for $a,p,q \in \mathbb{Z}$.
Here we remark that
$F$ on $E_r^{p,q}(K_{\mathcal{O}},W)$ above
is not necessarily a filtration by {\itshape subbundles}.
The morphism of $E_r$-terms is denoted by
$d_r^{p,q}$
for every $p,q,r$.
\end{say}

\begin{say}
For a triple $K$ as in \ref{para:1},
we consider the following conditions:
\begin{enumerate}
\item
\label{item:1}
The triple
\begin{gather}
(H^n(\Gr_m^WK_{\mathbb{R}}),
(H^n(\Gr_m^WK_{\mathcal{O}}),F), H^n(\Gr_m^W\alpha))
\end{gather}
is a variation of $\mathbb{R}$-Hodge structure
of weight $n+m$ on $X$
for every $m,n$.
\item
\label{item:4}
The spectral sequence associated to
$(\Gr_m^WK_{\mathcal{O}}, F)$
degenerates at $E_1$-terms for every $m$. 
\end{enumerate}
\end{say}

The following lemma is a counterpart
of \cite[Scholie (8.1.9)]{DeligneIII}
for variations of (mixed) Hodge structure.

\begin{lem}
\label{lem:3}
If $K$ satisfies \ref{item:1} and \ref{item:4},
then the following holds:
\begin{enumerate}
\item
\label{item:6}
$E_r^{p,q}(K,W)$
is a variation of $\mathbb{R}$-Hodge structure
of weight $q$ on $X$
for every $r=1, 2, \dots, \infty$
and for every $p,q$.
\item
The triple
\begin{equation}
(\Gr_m^WH^n(K_{\mathbb{R}}), (\Gr_m^WH^n(K_{\mathcal{O}}),F),
\Gr_m^WH^n(\alpha))
\end{equation}
is a variation of $\mathbb{R}$-Hodge structure of weight $n+m$
on $X$ for every $m, n \in \mathbb{Z}$.
\item
\label{item:11}
The spectral sequence associated
to $(K_{\mathcal{O}}, W)$ degenerates at $E_2$-terms.
\item
\label{item:17}
The spectral sequence associated to $(K_{\mathcal{O}}, F)$
degenerates at $E_1$-terms.
\item
\label{item:18}
There exists an isomorphism
\begin{equation}
E_r^{p,q}(\Gr_F^aK_{\mathcal{O}},W)
\simeq
\Gr_F^aE_r^{p,q}(K_{\mathcal{O}},W),
\end{equation}
under which the morphism of the $E_r$-terms of the left hand side
coincides with $\Gr_F^ad_r^{p,q}$ on the right hand side
for every $a,p,q,r$.
\item
\label{item:12}
The spectral sequence associated to $(\Gr_F^aK_{\mathcal{O}},W)$
degenerates at $E_2$-terms
for every $a \in \mathbb{Z}$.
\end{enumerate}
\end{lem}
\begin{proof}
From the assumption \ref{item:4}, 
the differential of $\Gr_m^WK_{\mathcal O}$ 
is strictly compatible with $F$ for every $m$.
Then we obtain the conclusions
from the assumption \ref{item:1}
by a similar argument to
\cite[Proposition (7.2.8) and Scholie (8.1.9)]{DeligneIII}
and by Lemma \ref{lem:2} below.
\end{proof}

\begin{lem}
\label{lem:2}
The category of the variations of $\mathbb{R}$-Hodge structure
of a fixed weight on a complex manifold
is an abelian category.
\end{lem}
\begin{proof}
By Lemma 3.14 (i) of \cite{fujino-fujisawa} 
(with $\mathbb{Q}$ replaced by $\mathbb{R}$), 
it suffices to prove that
the kernel and the cokernel of a morphism
of variations of $\mathbb{R}$-Hodge structure
satisfy the Griffiths transversality.
Let $\varphi \colon \mathbb{V}_1 \to \mathbb{V}_2$
be a morphism of variations of $\mathbb{R}$-Hodge structure
of the same weight on a complex manifold $X$.
Then, from the commutative diagram
\begin{equation}
\xymatrix{
0 \ar[r] & \mathcal{O}_X \otimes \kernel(\varphi) \ar[r] \ar[d]_{d \otimes \id}
& \mathcal{O}_X \otimes \mathbb{V}_1 \ar[r] \ar[r] \ar[d]_{d \otimes \id}
& \mathcal{O}_X \otimes \mathbb{V}_2 \ar[r] \ar[r] \ar[d]^{d \otimes \id}
& \mathcal{O}_X \otimes \coker(\varphi) \ar[r] \ar[r] \ar[d]^{d \otimes \id}
& 0 \\
0 \ar[r] & \Omega^1_X \otimes \kernel(\varphi) \ar[r]
& \Omega^1_X \otimes \mathbb{V}_1 \ar[r]
& \Omega^1_X \otimes \mathbb{V}_2 \ar[r]
& \Omega^1_X \otimes \coker(\varphi) \ar[r]
& 0, \\
}
\end{equation}
we can easily check that the first and the last vertical arrows
satisfy the Griffiths transversality.
Here we note that the Hodge filtrations
on $\mathcal{O}_X \otimes \kernel(\varphi)$
and $\mathcal{O}_X \otimes \coker(\varphi)$
are induced from
the Hodge filtrations
on $\mathcal{O}_X \otimes \mathbb{V}_1$
and $\mathcal{O}_X \otimes \mathbb{V}_2$ respectively.
\end{proof}

\begin{say}
\label{para:2}
Now we study the case
of a pair $(X, \Sigma)$,
where $X$ is a complex manifold
and $\Sigma$ a reduced simple normal crossing divisor on $X$
having finitely many irreducible components.
We set $X^*=X \setminus \Sigma$.
\end{say}

The following elementary lemma
will be used several times
in the present and next sections.

\begin{lem}
\label{c-lem3.1}
Let $\mathcal{F}$ and $\mathcal{G}$
be locally free $\mathcal{O}_X$-modules of finite rank on $X$
and $\varphi, \psi \colon \mathcal{F} \to \mathcal{G}$
morphisms of $\mathcal{O}_X$-modules.
If $\varphi|_{X^*}=\psi|_{X^*}$, then $\varphi=\psi$.
In particular, if $\varphi|_{X^*}=0$ then $\varphi=0$.
\end{lem}

\begin{say}
\label{para:4}
For the case of $(X, \Sigma)$,
we consider a triple
\begin{equation}
K=((K_{\mathbb{R}}, W), (K_{\mathcal{O}}, W, F), \alpha)
\end{equation}
consisting of
\begin{itemize}
\item
a bounded complex of $\mathbb{R}$-sheaves $K_{\mathbb{R}}$ on $X^*$
equipped with a finite increasing filtration $W$,
\item
a bounded complex of $\mathcal{O}_X$-modules $K_{\mathcal{O}}$ on $X$
equipped with a finite increasing filtration $W$
and a finite decreasing filtration $F$, and
\item
a morphism of complexes of $\mathbb{R}$-sheaves
$\alpha \colon K_{\mathbb{R}} \to K_{\mathcal{O}}|_{X^*}$
preserving the filtration $W$.
\end{itemize}
Note that only $(K_{\mathcal{O}},W,F)$ is defined over the whole $X$.
We set
\begin{equation}
K|_{X^*}=((K_{\mathbb{R}}, W), (K_{\mathcal{O}}, W, F)|_{X^*}, \alpha),
\end{equation}
which is a triple on $X^*$
considered in \ref{para:1},
and consider the following conditions:
\begin{enumerate}
\item
\label{item:7}
$K|_{X^*}$ satisfies \ref{item:1} and \ref{item:4} on $X^*$.
\item
For every $m,n$,
the $\mathbb{R}$-local system $H^n(\Gr_m^WK_{\mathbb{R}})$
is of quasi-unipotent local monodromies
around all the irreducible components of $\Sigma$.
\item
\label{item:9}
For every $m,n$,
the variation of $\mathbb{R}$-Hodge structure on $X^*$
in \ref{item:1} for $K|_{X^*}$
is polarizable.
\end{enumerate}
If $K$ satisfies these conditions,
then the lower canonical extension of
$\mathcal{O}_{X^*} \otimes H^n(\Gr_m^WK_{\mathbb{R}})
\simeq H^n(\Gr_m^WK_{\mathcal{O}})|_{X^*}$
is denoted by
${}^\ell H^n(\Gr_m^WK_{\mathcal{O}})|_{X^*}$
for every $m,n$.
By Schmid's nilpotent orbit theorem
\cite[(4.12)]{schmid},
the Hodge filtration $F$ on
$H^n(\Gr_m^WK_{\mathcal{O}})|_{X^*}$
extends to a filtration (by subbundles)
on ${}^\ell H^n(\Gr_m^WK_{\mathcal{O}})|_{X^*}$,
which is denoted by the same letter $F$.
Moreover, the condition \ref{item:4} for $K|_{X^*}$
implies that there exists a natural isomorphism
\begin{equation}
\label{eq:14}
H^n(\Gr_F^a\Gr_m^WK_{\mathcal{O}})|_{X^*}
\simeq
\Gr_F^aH^n(\Gr_m^WK_{\mathcal{O}})|_{X^*}
\end{equation}
for every $a,m,n$.
Here we remark that this morphism
induces the isomorphism in \ref{item:18} for $r=1$.

On the other hand,
we conclude \ref{item:6}--\ref{item:12} for $K|_{X^*}$
by Lemma \ref{lem:3}.
In addition, we can easily obtain the following:
\begin{enumerate}[resume]
\item
\label{item:14}
$H^n(K_{\mathbb{R}})$
is an $\mathbb{R}$-local system on $X^*$
of quasi-unipotent local monodromy along $\Sigma$
for every $n$.
\item
\label{item:16}
$W_mH^n(K_{\mathcal{O}})|_{X^*}
\simeq \mathcal{O}_{X^*} \otimes W_mH^n(K_{\mathbb{R}})$
for every $m,n$.
In particular, we have
$H^n(K_{\mathcal{O}})|_{X^*}
\simeq \mathcal{O}_{X^*} \otimes H^n(K_{\mathbb{R}})$
for every $n$.
\item
\label{item:15}
The variation of $\mathbb{R}$-Hodge structure
$E_r^{p,q}(K|_{X^*}, W)$ is polarizable
for every $p,q,r$.
In particular,
the variation of $\mathbb{R}$-Hodge structure
\begin{equation}
(\Gr_m^WH^n(K_{\mathbb{R}}), (\Gr_m^WH^n(K_{\mathcal{O}}),F)|_{X^*},
\Gr_m^WH^n(\alpha))
\end{equation}
is polarizable.
\end{enumerate}
\end{say}

The following lemma and theorem play essential roles
in the proofs of Theorems \ref{z-thm1.1} and \ref{z-thm1.4}.

\begin{lem}
\label{lem:1}
Let $(X, \Sigma)$ be as in \ref{para:2}
and $K=((K_{\mathbb{R}}, W), (K_{\mathcal{O}}, W, F), \alpha)$
as in \ref{para:4}
satisfying the conditions
\ref{item:7}--\ref{item:9}.
For $a \in \mathbb{Z}$,
if there exists an isomorphism
\begin{equation}
\label{eq:15}
H^n(\Gr_F^a\Gr_m^WK_{\mathcal{O}})
\overset{\simeq}{\longrightarrow}
\Gr_F^a({}^\ell H^n(\Gr_m^WK_{\mathcal{O}})|_{X^*})
\end{equation}
whose restriction to $X^*$
coincides with the isomorphism \eqref{eq:14}
for every $m,n$, 
then the spectral sequence
associated to
$(\Gr_F^aK_{\mathcal{O}},W)$
degenerates at $E_2$-terms
and the morphism of $E_1$-terms 
splits.
Moreover $\Gr_m^WH^n(\Gr_F^aK_{\mathcal{O}})$
is locally free of finite rank for every $m,n$.
In particular, so is
$H^n(\Gr_F^aK_{\mathcal{O}})$ for every $n$.
\end{lem}
\begin{proof}
Because the variation of $\mathbb{R}$-Hodge structure
$E_r^{p,q}(K|_{X^*},W)$
is of quasi-unipotent local monodromy and polarizable
as mentioned in \ref{item:16},
we obtain a filtered $\mathcal{O}_X$-module
$({}^\ell E_r^{p,q}(K_{\mathcal{O}},W)|_{X^*}, F)$
for every $p,q,r$
by using Schmid's nilpotent orbit theorem
as in \ref{para:4}.

Because $d_1^{p,q}|_{X^*}$
is a morphism of variations of $\mathbb{R}$-Hodge structure of weight $q$,
we obtain a polarizable variation of $\mathbb{R}$-Hodge structure
$I^{p,q}$ of weight $q$ on $X^*$
by
\begin{equation}
\begin{split}
I^{p,q}
&=(I_{\mathbb{R}}^{p,q}, (I_{\mathcal{O}}^{p,q},F)) \\
&=\image(d_1^{p-1,q}|_{X^*} \colon
E_1^{p-1,q}(K,W)|_{X^*}
\to E_1^{p,q}(K,W)|_{X^*})
\end{split}
\end{equation}
for every $p,q$.
It is clear that
$I_{\mathbb{R}}^{p,q}$ is of quasi-unipotent local monodromy along $\Sigma$.
Thus we obtain a filtered $\mathcal{O}_X$-module
$({}^\ell I_{\mathcal{O}}^{p,q},F)$
as in \ref{para:4} again.

By the semisimplicity of the polarizable variations
of $\mathbb{R}$-Hodge structure,
we have the direct sum decomposition
\begin{equation}
E_1^{p,q}(K|_{X^*},W)
\simeq
E_2^{p,q}(K|_{X^*},W) \oplus I^{p,q} \oplus I^{p+1,q}
\end{equation}
as variations of $\mathbb{R}$-Hodge structure on $X^*$.
Therefore we obtain the direct sum decomposition
\begin{equation}
\label{eq:11}
({}^\ell E_1^{p,q}(K_{\mathcal{O}},W)|_{X^*},F)
\simeq
({}^\ell E_2^{p,q}(K_{\mathcal{O}},W)|_{X^*},F)
\oplus ({}^\ell I_{\mathcal{O}}^{p,q},F)
\oplus ({}^\ell I_{\mathcal{O}}^{p+1,q},F)
\end{equation}
as filtered $\mathcal{O}_X$-modules
because the extension of the filtration
is unique by \cite[Corollary 5.2]{fujino-fujisawa}.
It is clear that the morphism ${}^\ell d_1^{p,q}|_{X^*}$
is identified with the composite of
the projection
\begin{equation}
{}^\ell E_2^{p,q}(K_{\mathcal{O}},W)|_{X^*}
\oplus {}^\ell I_{\mathcal{O}}^{p,q}
\oplus {}^\ell I_{\mathcal{O}}^{p+1,q}
\to {}^\ell I_{\mathcal{O}}^{p+1,q}
\end{equation}
and the inclusion
\begin{equation}
{}^\ell I_{\mathcal{O}}^{p+1,q}
\hookrightarrow {}^\ell E_2^{p+1,q}(K_{\mathcal{O}},W)|_{X^*}
\end{equation}
under the isomorphism \eqref{eq:11}.

We fix $a \in \mathbb{Z}$
satisfying the assumption,
and consider the spectral sequence 
\begin{equation}
\label{eq:9}
E_r^{p,q}(\Gr_F^aK_{\mathcal{O}}, W)
\Rightarrow
H^{p+q}(\Gr_F^aK_{\mathcal{O}})
\end{equation}
with the morphism of $E_r$-terms
$d_r^{p,q}(\Gr_F^aK_{\mathcal{O}},W)$.
By the assumption,
there exists an isomorphism
\begin{equation}
\label{eq:12}
E_1^{p,q}(\Gr_F^aK_{\mathcal{O}}, W)
\simeq
\Gr_F^a({}^\ell E_1^{p,q}(K_{\mathcal{O}},W)|_{X^*})
\end{equation}
for every $p,q$,
because we have
\begin{equation}
\begin{split}
E_1^{p,q}(\Gr_F^aK_{\mathcal{O}}, W)
&\simeq
H^{p+q}(\Gr_{-p}^W\Gr_F^aK_{\mathcal{O}}) \\
&\simeq
H^{p+q}(\Gr_F^a\Gr_{-p}^WK_{\mathcal{O}}) \\
&\simeq
\Gr_F^a({}^\ell H^{p+q}(\Gr_{-p}^WK_{\mathcal{O}})|_{X^*}) \\
&\simeq
\Gr_F^a({}^\ell E_1^{p,q}(K_{\mathcal{O}},W)|_{X^*}).
\end{split}
\end{equation}
In particular,
$E_1^{p,q}(\Gr_F^aK_{\mathcal{O}}, W)$
is locally free of finite rank for every $p,q$.
Moreover, the restriction of the isomorphism \eqref{eq:12} to $X^*$
identifies
$d_1^{p,q}(\Gr_F^aK_{\mathcal{O}},W)|_{X^*}$
and $\Gr_F^ad_1^{p,q}|_{X^*}$
by the assumption and by \ref{item:18} for $K|_{X^*}$.
Then
\begin{equation}
d_1^{p,q}(\Gr_F^aK_{\mathcal{O}},W)
=
\Gr_F^a({}^\ell d_1^{p,q}|_{X^*})
\end{equation}
under the identification \eqref{eq:12}
by Lemma \ref{c-lem3.1}. 
Therefore there exists an isomorphism
\begin{equation}
E_2^{p,q}(\Gr_F^aK_{\mathcal{O}}, W)
\simeq
\Gr_F^a({}^\ell E_2^{p,q}(K_{\mathcal{O}},W)|_{X^*})
\end{equation}
and $E_2^{p,q}(\Gr_F^aK_{\mathcal{O}}, W)$
is isomorphic to a direct factor of
$E_1^{p,q}(\Gr_F^aK_{\mathcal{O}}, W)$
by the direct sum decomposition \eqref{eq:11}.
In particular,
$E_2^{p,q}(\Gr_F^aK_{\mathcal{O}}, W)$
is locally free of finite rank.
Then $d_2^{p,q}(\Gr_F^aK_{\mathcal{O}}, W)=0$
because $d_2^{p,q}(\Gr_F^aK_{\mathcal{O}}, W)|_{X^*}=0$
by \ref{item:12} of Lemma \ref{lem:3}.
Inductively,
$E_r^{p,q}(\Gr_F^aK_{\mathcal{O}}, W)$
is locally free of finite rank, and
$d_r^{p,q}(\Gr_F^aK_{\mathcal{O}}, W)=0$
for $r \ge 2$
by \ref{item:12} of Lemma \ref{lem:3} again.
Thus the spectral sequence \eqref{eq:12}
degenerates at $E_2$-terms.
Moreover,
\begin{equation}
\Gr_m^WH^n(\Gr_F^aK_{\mathcal{O}})
\simeq
E_2^{-m,n+m}(\Gr_F^aK_{\mathcal{O}}, W)
\simeq
\Gr_F^a({}^\ell E_2^{-m,n+m}(K_{\mathcal{O}},W)|_{X^*})
\end{equation}
is locally free of finite rank for every $m,n$.
\end{proof}

\begin{thm}
\label{lem:4}
Let $(X, \Sigma)$ be as in \ref{para:2}
and $K=((K_{\mathbb{R}}, W), (K_{\mathcal{O}}, W, F), \alpha)$
as in \ref{para:4}
satisfying the conditions
\ref{item:7}--\ref{item:9}.
If there exists an isomorphism
\begin{equation}
\label{eq:3}
H^n(\Gr_m^WK_{\mathcal{O}})
\simeq {}^\ell H^n(\Gr_m^WK_{\mathcal{O}})|_{X^*}
\end{equation}
whose restriction to $X^*$
is the identity
for every $m,n$,
then we have the following:
\begin{enumerate}
\item
\label{item:8}
There exist isomorphisms
\begin{gather}
H^n(K_{\mathcal{O}})
\simeq {}^\ell H^n(K_{\mathcal{O}})|_{X^*}, \\
W_mH^n(K_{\mathcal{O}})
\simeq {}^\ell W_mH^n(K_{\mathcal{O}})|_{X^*},
\end{gather}
whose restriction to $X^*$ coincide with the identities,
for every $m, n \in \mathbb{Z}$.
In particular,
$\Gr_m^WH^n(K_{\mathcal{O}})$
is locally free of finite rank on $X$ for every $m,n$.
\item
\label{item:2}
The spectral sequence associated to $(K_{\mathcal{O}}, W)$ 
degenerates at $E_2$-terms
on $X$.
\end{enumerate}

If we further assume that
$H^n(\Gr_F^a\Gr_m^WK_{\mathcal{O}})$ is locally free of finite rank
for every $a,m,n$
and that
$K_{\mathcal{O}}$ satisfies \ref{item:4}
on the whole $X$,
then we have the following:
\begin{enumerate}[resume]
\item
\label{item:3}
The spectral sequence associated to $(K_{\mathcal{O}},F)$
degenerates at $E_1$-terms.
\item
\label{item:5}
$\Gr_F^a\Gr_m^WH^n(K_{\mathcal{O}})$
is a locally free $\mathcal{O}_X$-module of finite rank
for every $a,m,n$.
\item
\label{item:10}
The spectral sequence associated to $(\Gr_F^aK_{\mathcal{O}},W)$
degenerates at $E_2$-terms for every $a$.
\end{enumerate}
\end{thm}
\begin{proof}
We use the same notation as in the proof of Lemma \ref{lem:1}.

From \eqref{eq:3}, we have
\begin{equation}
\label{eq:5}
E_1^{p,q}(K_{\mathcal{O}},W)
\simeq {}^\ell E_1^{p,q}(K_{\mathcal{O}},W)|_{X^*}
\end{equation}
for every $p,q$.
In particular, $E_1^{p,q}(K_{\mathcal{O}},W)$
is locally free of finite rank.
Then, by Lemma \ref{c-lem3.1},
we have $d_1^{p,q}={}^\ell d_1^{p,q}|_{X^*}$
under the isomorphism \eqref{eq:5}
because $({}^\ell d_1^{p,q}|_{X^*})|_{X^*}=d_1^{p,q}|_{X^*}$.
Therefore
\begin{equation}
\label{eq:2}
E_2^{p,q}(K_{\mathcal{O}},W) \simeq {}^\ell E_2^{p,q}(K_{\mathcal{O}},W)|_{X^*}
\end{equation}
by \eqref{eq:11}.
In particular, $E_2^{p,q}(K_{\mathcal{O}},W)$ is locally free of finite rank
for every $p,q$.
Lemma \ref{c-lem3.1} implies $d_2^{p,q}=0$
because $d_2^{p,q}|_{X^*}=0$
by \ref{item:11} of Lemma \ref{lem:3}.
Inductively, $E_r^{p,q}(K_{\mathcal{O}},W)$ is locally free of finite rank
and $d_r^{p,q}=0$ for $r \ge 2$.
Thus we obtain \ref{item:2}.
We have
\begin{equation}
\Gr_m^WH^n(K_{\mathcal{O}})
\simeq
E_2^{-m,n+m}(K_{\mathcal{O}},W)
\simeq
{}^\ell E_2^{-m,n+m}(K_{\mathcal{O}},W)|_{X^*}
\simeq
{}^\ell(\Gr_m^WH^n(K_{\mathcal{O}})|_{X^*}),
\end{equation}
from which we obtain \ref{item:8}.

Next, we will prove the latter half of the theorem.
By the assumption \ref{item:4} on the whole $X$,
the sequence of the canonical morphisms
\begin{equation}
\label{eq:6}
0 \to H^n(F^{a+1}\Gr_m^WK_{\mathcal{O}})
\to H^n(F^a\Gr_m^WK_{\mathcal{O}})
\to H^n(\Gr_F^a\Gr_m^WK_{\mathcal{O}})
\to 0
\end{equation}
is exact for every $a,m,n$.
Then we have
\begin{equation}
\Gr_F^aH^n(\Gr_m^WK_{\mathcal{O}})
\simeq
H^n(\Gr_F^a\Gr_m^WK_{\mathcal{O}})
\end{equation}
for every $a,m,n$,
from which $\Gr_F^aH^n(\Gr_m^WK_{\mathcal{O}})$
is locally free of finite rank
by the assumption.
Thus $F^aH^n(\Gr_m^WK_{\mathcal{O}})$
turns out to be a subbundle for every $a$.
Since the extension of the filtration $F$
on $E_1^{p,q}(K_{\mathcal{O}},W)|_{X^*}$
is unique by \cite[Corollary 5.2]{fujino-fujisawa},
the isomorphism \eqref{eq:3}
induces an isomorphism
\begin{equation}
\label{eq:7}
(H^n(\Gr_m^WK_{\mathcal{O}}), F)
\simeq
({}^\ell H^n(\Gr_m^WK_{\mathcal{O}})|_{X^*}, F)
\end{equation}
of filtered $\mathcal{O}_X$-modules.
In particular, we have isomorphisms
\begin{equation}
H^n(\Gr_F^a\Gr_m^WK_{\mathcal{O}})
\simeq
\Gr_F^aH^n(\Gr_m^WK_{\mathcal{O}})
\simeq
\Gr_F^a({}^\ell H^n(\Gr_m^WK_{\mathcal{O}})|_{X^*})
\end{equation}
whose restriction to $X^*$
coincides with the natural isomorphism \eqref{eq:14}.
Thus we obtain \ref{item:10}
by Lemma \ref{lem:1}.
From \eqref{eq:7},
it turns out that
\eqref{eq:5} is, in fact,
an isomorphism
of filtered $\mathcal{O}_X$-modules
\begin{equation}
\label{eq:8}
(E_1^{p,q}(K_{\mathcal{O}},W), F)
\simeq ({}^\ell E_1^{p,q}(K_{\mathcal{O}},W)|_{X^*}, F),
\end{equation}
under which $d_1^{p,q}={}^\ell d_1^{p,q}|_{X^*}$.
Therefore $d_1^{p,q}$ is strictly compatible with $F$
by the direct sum decomposition \eqref{eq:11}.
The assumption \ref{item:4}
implies that
\begin{equation}
d_0^{p,q}
\colon E_0^{p,q}(K_{\mathcal{O}},W)
\to E_0^{p,q+1}(K_{\mathcal{O}},W)
\end{equation}
is strictly compatible with the filtration $F$.
By \ref{item:2}, we already have $d_r^{p,q}=0$ for all $r \ge 2$.
Thus the morphism $d_r^{p,q}$ is strictly compatible
with the filtration $F$ for every $p,q,r$.
Hence we conclude \ref{item:3}
by the lemma on two filtrations
(see e.g.~\cite[Proposition (7.2.8)]{DeligneIII}, 
\cite[Theorem 3.12, 3)]{Peters-SteenbrinkMHS}).

From \eqref{eq:8}
and the direct sum decomposition \eqref{eq:11},
the isomorphism \eqref{eq:2} induces
an isomorphism of filtered $\mathcal{O}_X$-modules
\begin{equation}
(E_2^{p,q}(K_{\mathcal{O}},W), F_{\rec})
\simeq ({}^\ell E_2^{p,q}(K_{\mathcal{O}},W)|_{X^*}, F)
\end{equation}
where $F_{\rec}$ denotes the inductive filtration
(la filtration r\'ecurrente in \cite[(1.3.11)]{DeligneII}).
On the other hand,
$F=F_{\rec}$ on $E_2^{p,q}(K_{\mathcal{O}},W)$
by the lemma on two filtrations again
(see e.g.~\cite[Proposition (7.2.5)]{DeligneIII}, 
\cite[Theorem 3.12, 1)]{Peters-SteenbrinkMHS}.
Thus we have
\begin{equation}
\Gr_F^aE_2^{p,q}(K_{\mathcal{O}},W)
\simeq
\Gr_{F_{\rec}}^aE_2^{p,q}(K_{\mathcal{O}},W)
\simeq
\Gr_F^a({}^\ell E_2^{p,q}(K_{\mathcal{O}},W)|_{X^*})
\end{equation}
for every $a,p,q$.
In particular
\begin{equation}
\Gr_F^aE_2^{p,q}(K_{\mathcal{O}},W)
\end{equation}
is locally free of finite rank for every $a,p,q$.
Moreover, the filtration $F$ on $E_2^{p,q}(K_{\mathcal{O}},W)$
coincides with the filtration $F$ on $\Gr_{-p}^WH^{p+q}(K_{\mathcal{O}})$
under the isomorphism
\begin{equation}
E_2^{p,q}(K_{\mathcal{O}},W)
\simeq
\Gr_{-p}^WH^{p+q}(K_{\mathcal{O}}) 
\end{equation}
by the lemma on two filtrations
(see e.g.~\cite[Theorem 3.12, 2)]{Peters-SteenbrinkMHS}).
Thus we obtain \ref{item:5} from
\begin{equation}
\Gr_F^aE_2^{-m,n+m}(K_{\mathcal{O}})
\simeq
\Gr_F^a\Gr_m^WH^n(K_{\mathcal{O}})
\end{equation}
for every $a,m,n$.
\end{proof}

\section{Proofs of Theorems \ref{z-thm1.1}, \ref{z-thm1.3}, and \ref{z-thm1.4}}
\label{c-sec3}

In this section, we will prove Theorems \ref{z-thm1.1}, 
\ref{z-thm1.3}, and 
\ref{z-thm1.4}. 
As already mentioned in Remark \ref{z-rem1.6}, 
our approach to 
Theorem \ref{z-thm1.1} (ii)--(iv) here is different from
\cite{fujino-fujisawa} (see also \cite[Section 13]{fujino-basic-slc}) 
because we do not assume that $(X, D)$ is projective over $Y$
in this section.
We use the notation in \cite[Section 4]{fujino-fujisawa}.

\begin{say}
\label{para:3} 
First, we briefly recall several constructions and results
in \cite[Section 4]{fujino-fujisawa},
which are necessary for the proofs of Theorems \ref{z-thm1.1}
and \ref{z-thm1.4}.

Let $f\colon (X, D)\to Y$ be as in Theorems 
\ref{z-thm1.1} and \ref{z-thm1.4}. 
Let 
\begin{equation*}
X=\bigcup_{i \in I}X_i \quad {\text{and}} \quad 
D=\bigcup_{\lambda \in \Lambda}D_{\lambda}
\end{equation*} 
be the irreducible decompositions of $X$ and $D$, respectively.
Fixing orders $<$ on $\Lambda$ and $I$,
we put 
\begin{equation*}
D_k \cap X_l
=\coprod_{\substack{\lambda_0< \lambda_1 < \cdots < \lambda_k \\
            i_0 < i_1 < \cdots < i_l}}
D_{\lambda_0} \cap D_{\lambda_1} \cap \cdots \cap D_{\lambda_k}
\cap X_{i_0} \cap X_{i_1} \cap \cdots \cap X_{i_l}
\end{equation*}
for $k,l \ge 0$ (see \cite[4.14]{fujino-fujisawa}). 
Here we use the convention 
\begin{align*}
&D_k=D_k \cap X_{-1}
=\coprod_{\lambda_0 < \lambda_1 < \cdots < \lambda_k}
D_{\lambda_0} \cap D_{\lambda_1} \cap \cdots \cap D_{\lambda_k} \\
&X_l=D_{-1} \cap X_l
=\coprod_{i_0 < i_1 < \cdots < i_l}
X_{i_0} \cap X_{i_1} \cap \cdots \cap X_{i_l}
\end{align*}
for $k,l \ge 0$. 
By setting
\begin{equation}
(X, D)_n
:=(D \cap X)_n \setminus D_n
=\coprod_{\substack{k+l+1=n \\ l \ge 0}}D_k \cap X_l,
\end{equation}
we obtain an augmented semisimplicial variety
$\varepsilon \colon (X, D)_{\bullet} \to X$.
Note that $(X, D)_n$ is
the disjoint union of all the strata of $(X,D)$
of dimension $\dim X-n$
for all $n \in \mathbb{Z}_{\ge 0}$.
We set $f_n:=f \varepsilon_n \colon (X, D)_n \to Y$
for every $n$.
Then $f_n$ is smooth over $Y^*=Y \setminus \Sigma$.
Then the complex $\varepsilon_*\mathbb{R}_{(X, D)_{\bullet}}$
is given by
\begin{equation*}
(\varepsilon_*\mathbb{R}_{(X, D)_{\bullet}})^n
=(\varepsilon_n)_*\mathbb{R}_{(X, D)_n}
=\bigoplus_{l \ge 0}\mathbb{R}_{D_{n-l-1} \cap X_l}
\end{equation*}
with the \v{C}ech type morphism $\delta$ as the differential.
Note that this complex is the single complex associated to
the double complex obtained
by deleting the first vertical column
of the double complex in \cite[p.626, 4.14]{fujino-fujisawa},
and by replacing $\mathbb{Q}$ with $\mathbb{R}$.
Then we have quasi-isomorphisms
\begin{equation*}
i_!\mathbb{R}_{X \setminus D}
\xrightarrow{\simeq}
(\mathbb{R}_X
\to
\mathbb{R}_{D_0}
\xrightarrow{\delta}
\mathbb{R}_{D_1}
\xrightarrow{\delta}
\cdots) \\
\xrightarrow{\simeq}
\varepsilon_*\mathbb{R}_{(X, D)_{\bullet}}
\end{equation*}
from the double complex in \cite{fujino-fujisawa} mentioned above,
where $i$ denotes 
the open immersion $X \setminus D \hookrightarrow X$.
By setting
\begin{equation*}
L_m(\varepsilon_*\mathbb{R}_{(X, D)_{\bullet}})^n
=
\begin{cases}
0 \quad &n < -m \\
(\varepsilon_n)_*\mathbb{R}_{(X, D)_n} \quad &n \ge -m
\end{cases}
\end{equation*}
a finite increasing filtration $L$ is defined
on $\varepsilon_*\mathbb{R}_{(X, D)_{\bullet}}$.
We have the relative de Rham complex
$\Omega_{(X, D)_{\bullet}/Y}$
for the morphism $f \varepsilon \colon (X, D)_{\bullet} \to Y$.
Then the complex
$\varepsilon_*\Omega_{(X, D)_{\bullet}/Y}$
is given by
\begin{equation}
\label{eq-total-deRham}
(\varepsilon_*\Omega_{(X, D)_{\bullet}/Y})^n
=\bigoplus_{k \ge 0}
(\varepsilon_k)_*\Omega^{n-k}_{(X, D)_k/Y}
\end{equation}
with the differential
$\delta+(-1)^kd$ on $(\varepsilon_k)_*\Omega^{n-k}_{(X, D)_k/Y}$,
where $\delta$ denotes the \v{C}ech type morphism for $(X, D)_{\bullet}$
and $d$ denotes the differential of the relative de Rham complex
$\Omega_{(X, D)_n/Y}$.
By setting
\begin{gather}
L_m(\varepsilon_*\Omega_{(X, D)_{\bullet}/Y})^n
=\bigoplus_{k \ge -m}
(\varepsilon_k)_*\Omega^{n-k}_{(X, D)_k/Y}
\label{eq-filtration-L} \\
F^p(\varepsilon_*\Omega_{(X, D)_{\bullet}/Y})^n
=\bigoplus_{0 \le k \le n-p}
(\varepsilon_k)_*\Omega^{n-k}_{(X, D)_k/Y},
\label{eq-filtration-F}
\end{gather}
a finite increasing filtration $L$
and a finite decreasing filtration $F$
on $\varepsilon_*\Omega_{(X, D)_{\bullet}/Y}$
are defined.
The canonical morphism
$\mathbb{R}_{(X, D)_n} \to \mathcal{O}_{(X,D)_n}$
induces a morphism of complexes
$\iota \colon \varepsilon_*\mathbb{R}_{(X,D)_{\bullet}}
\to \varepsilon_*\Omega_{(X, D)_{\bullet}/Y}$.

By setting
\begin{equation}
\label{eq:1}
\begin{split}
K
&=((K_{\mathbb{R}}, L), (K_{\mathcal{O}}, L, F), \alpha) \\
&=((Rf_*\varepsilon_*\mathbb{R}_{(X, D)_{\bullet}}, L)|_{Y^*},
(Rf_*\varepsilon_*\Omega_{(X, D)_{\bullet}/Y}, L, F),
Rf_*\iota|_{Y^*})
\end{split}
\end{equation} 
(see \cite[4.1]{fujino-fujisawa}), 
we obtain a triple $K$ as in \ref{para:4}
satisfying the following:
\begin{enumerate}
\item
\label{qis-Rstr}
There exists a quasi-isomorphism
$R(f|_{X^* \setminus D^*})_{!}\mathbb{R}_{X^* \setminus D^*}
\simeq K_{\mathbb{R}}$.
\item
\label{qis-O(-D)}
There exists a quasi-isomorphism
$\Gr_F^pK_{\mathcal{O}}
\simeq
Rf_*\varepsilon_*\Omega^p_{(X, D)_{\bullet}/Y}[-p]$.
for every $p$.
In particular, $Rf_*\mathcal{O}_X(-D) \simeq \Gr_F^0K_{\mathcal{O}}$.
\item
\label{b-gr_m}
For every $m \in \mathbb{Z}$,
\begin{equation*}
\begin{split}
\Gr_m^LK
&=(\Gr_m^LK_{\mathbb{R}}, (\Gr_m^LK_{\mathcal{O}}, F), \Gr_m^L\alpha) \\
&\simeq
\bigoplus_S
(R(f_{S^*})_*\mathbb{R}_{S^*}[m],
(R(f_S)_*\Omega_{S/Y}[m], F),
R(f_{S^*})_*\iota_{S^*}[m]),
\end{split}
\end{equation*}
where
$S$ runs through all $(\dim X+m)$-dimensional
strata of $(X, D)$, $f_S=f|_S$, 
$S^*=f_S^{-1}(Y^*)$,
$f_{S^*}=(f_S)|_{S^*}$,
and $\iota_{S^*}$ is the composite
$\mathbb{R}_{S^*}
\hookrightarrow \mathbb{C}_{S^*}
\to \Omega_{S^*/Y^*}$.
\end{enumerate} 
\end{say}

First, we will prove Theorem \ref{z-thm1.4}.

\begin{proof}[Proof of Theorem \ref{z-thm1.4}] 
We use the notations and terminologies 
in \ref{para:3}.
We will prove that
the spectral sequence
\begin{equation}
\label{b-eq-ss1}
E_r^{p,q}(\Gr_F^0K_{\mathcal{O}}, L)
\Rightarrow
H^{p+q}(\Gr_F^0K_{\mathcal{O}})
\end{equation}
associated to the filtered complex
$(\Gr_F^0K_{\mathcal{O}}, L)$
satisfies the desired properties.

By \ref{qis-O(-D)},
we have an isomorphism
\begin{equation}
H^{p+q}(\Gr_F^0K_{\mathcal{O}})
\simeq
R^{p+q}f_*\mathcal{O}_X(-D)
\end{equation}
as desired.

By \ref{b-gr_m},
$K$ satisfies
\ref{item:7}--\ref{item:9}
because $f_S \colon S \to Y$ is smooth over $Y^*$
for every stratum $S$ of $(X, D)$.
Moreover, we have
\begin{equation}
H^n(\Gr_F^0\Gr_m^LK_{\mathcal{O}})
\simeq
\bigoplus_S
R^{n+m}(f_S)_*\mathcal{O}_S
\end{equation}
where
$S$ runs through all $(\dim X+m)$-dimensional
strata of $(X, D)$.
Therefore $K$ satisfies the assumption of Lemma \ref{lem:1}
by the dual of Theorem \ref{z-thm2.14}.
Thus we obtain the conclusion by Lemma \ref{lem:1}.
\end{proof}

Next, we will prove Theorem \ref{z-thm1.1} (i).

\begin{proof}[{Proof of Theorem \ref{z-thm1.1} {\em{(i)}}}]
As already mentioned above,
$K$ satisfies the conditions
\ref{item:7}--\ref{item:9}.
By applying Lemma \ref{lem:3}
together with \ref{item:14}--\ref{item:15},
the triple
\begin{equation}
(\Gr_m^LH^k(K_{\mathbb{R}}),
(\Gr_m^LH^k(K_{\mathcal{O}}), F)|_{Y^*},
\Gr_m^LH^k(\alpha))
\end{equation}
is a polarizable variation of $\mathbb{R}$-Hodge structure
of weight $k+m$ for every $k,m$.
By \ref{qis-Rstr}, we have
$R^k(f|_{X^* \setminus D^*})_{!}\mathbb{R}_{X^* \setminus D^*}
\simeq H^k(K_{\mathbb{R}})$,
which implies
$\mathcal{V}^k_{Y^*} \simeq H^k(K_{\mathcal{O}})|_{Y^*}$
as in \ref{item:16}
for all $k$.
By using these isomorphisms,
we introduce filtrations
$L$ on $R^k(f|_{X^* \setminus D^*})_{!}\mathbb{R}_{X^* \setminus D^*}$
and $\mathcal{V}^k_{Y^*}$,
$F$ on $\mathcal{V}^k_{Y^*}$,
and then obtain
a polarizable variation of $\mathbb{R}$-Hodge structure
\begin{equation}
(\Gr_m^{L[k]}R^k(f|_{X^* \setminus D^*})_{!}\mathbb{R}_{X^* \setminus D^*},
(\Gr_m^{L[k]}\mathcal{V}^k_{Y^*},F), 
\Gr_m^{L[k]}\beta)
\end{equation}
of weight $m$ on $Y^*$ for every $k,m$, 
where the natural morphism
$R^k(f|_{X^* \setminus D^*})_{!}\mathbb{R}_{X^* \setminus D^*}
\to
\mathcal{V}^k_{Y^*}
=R^k(f|_{X^* \setminus D^*})_{!}\mathbb{R}_{X^* \setminus D^*}
\otimes \mathcal{O}_{Y^*}$
is denoted by $\beta$.

On the other hand,
the Griffiths transversality for
\begin{equation}
(\mathcal{V}^k_{Y^*}, F)
\simeq
(H^k(K_{\mathcal{O}}), F)|_{Y^*}
\simeq
(R^kf_*\varepsilon_*\Omega_{(X, D)_{\bullet}/Y}, F)|_{Y^*}
\end{equation}
can be easily seen
by the same way as in the proof of Lemma 4.5 of \cite{fujino-fujisawa}.
Therefore
\begin{equation}
((R^k(f|_{X^* \setminus D^*})_{!}\mathbb{R}_{X^* \setminus D^*},L[k]),
(\mathcal{V}^k_{Y^*},L[k],F), \beta)
\end{equation}
is a graded polarizable variation
of $\mathbb{R}$-mixed Hodge structure on $Y^*$ as desired.
\end{proof}

In order to prove Theorem \ref{z-thm1.1} (ii)--(iv),
we recall results in \cite{steenbrink1} and \cite{steenbrink2}
in a slightly generalized form.

\begin{defn}\label{c-def3.2}
Let $f \colon X \to Y$ be a surjective morphism of smooth complex varieties
and $\Sigma$ a simple normal crossing divisor on $Y$.
We assume that $E=(f^*\Sigma)_{\red}$ is
a simple normal crossing divisor on $X$.
For such $f$,
we set
\begin{equation*}
\Omega^1_{X/Y}(\log E)
=
\coker(f^*\Omega^1_Y(\log \Sigma) \to \Omega^1_X(\log E))
\end{equation*}
and
\begin{equation*}
\Omega^p_{X/Y}(\log E)=\bigwedge^p\Omega^1_{X/Y}(\log E)
\end{equation*}
for every $p$.
An $f^{-1}\mathcal{O}_Y$-differential
$d \colon \Omega^p_{X/Y}(\log E) \to \Omega^{p+1}_{X/Y}(\log E)$
can be uniquely defined by the commutative diagram 
\begin{equation*}
\xymatrix{
\Omega^p_X(\log E) \ar[r]\ar[d]_-d& \Omega^p_{X/Y}(\log E)\ar[d]^-d \\
\Omega^{p+1}_X(\log E) \ar[r]& \Omega^{p+1}_{X/Y}(\log E),
}
\end{equation*} 
where the horizontal arrows are the canonical surjections
induced from the surjection
$\Omega^1_X(\log E) \to \Omega^1_{X/Y}(\log E)$.
Thus we obtain a complex of $f^{-1}\mathcal{O}_Y$-modules
$\Omega_{X/Y}(\log E)$,
which is called the relative log de Rham complex of $f$.
\end{defn}

\begin{lem}
\label{c-lem3.3}
Let $f \colon X \to Y$ be a proper surjective morphism
from a K\"ahler manifold $X$ to a smooth complex variety $Y$.
Assume that
there exists a smooth divisor $\Sigma$ on $Y$ such that
\begin{enumerate}
\item
\label{St2-smooth-overY*}
$f$ is smooth over $Y^*=Y \setminus \Sigma$,
\item
\label{St2-snc}
$E=(f^*\Sigma)_{\red}$ is a simple normal crossing divisor on $X$
having finitely many irreducible components, and
\item
\label{log-smoothness}
$\Omega_{X/Y}^1(\log E)$
is a locally free $\mathcal{O}_X$-module
of finite rank.
\end{enumerate}
Then we have
\begin{equation*}
R^if_*\Omega_{X/Y}(\log E)
\simeq
{}^l(R^if_*\Omega_{X/Y}(\log E)|_{Y^*})
\simeq 
{}^l(\mathcal{O}_{Y^*} \otimes (R^if_*\mathbb{C}_X)|_{Y^*}) 
\end{equation*}
for all $i$,
where ${}^l(\cdot)$ stands for the lower canonical extension as before.
In particular, $R^if_*\Omega_{X/Y}(\log E)$
is a locally free $\mathcal{O}_Y$-module of finite rank
for all $i$.
Moreover, $R^if_*\Omega_{X/Y}^p(\log E)$
is also a locally free $\mathcal{O}_Y$-module of finite rank,
and the stupid filtration
$($filtration b\^ete in \cite[(1.4.7)]{DeligneII}$)$
$F$ on $\Omega_{X/Y}(\log E)$
induces the natural exact sequence
\begin{equation}
\label{eq-exact-steenbrink}
0 \to R^if_*F^{p+1}\Omega_{X/Y}(\log E)
\to R^if_*F^p\Omega_{X/Y}(\log E)
\to R^if_*\Omega^p_{X/Y}(\log E)
\to 0
\end{equation}
for all $i, p$.
\end{lem}
\begin{proof}
We may assume
$Y=\Delta^k$ with the coordinates $t_1, \dots, t_k$
and $\Sigma=\{t_1=0\}$. 
For any $x \in E$,
we can take local coordinates $x_1, \dots, x_n$
centered at $x$ on $X$
with 
\begin{equation}
f^*t_1=x_1^{a_1} \cdots x_l^{a_l}
\end{equation}
for some $a_1, \dots, a_l \in \mathbb{Z}_{>0}$
by \ref{St2-snc}.
We set $f_i=f^*t_i$ for $i=2, \dots, k$.
On the other hand,
we have the canonical exact sequence
\begin{equation}
\label{eq-exact-seq-logsmooth}
0 \to f^*\Omega_Y^1(\log \Sigma)_x \otimes \mathbb{C}(x)
\to \Omega_X^1(\log E)_x \otimes \mathbb{C}(x)
\to \Omega_{X/Y}^1(\log E)_x \otimes \mathbb{C}(x)
\to 0,
\end{equation}
where $\mathbb{C}(x)$ denotes the residue field at $x$,
because $\Omega_{X/Y}^1(\log E)$ is a locally free $\mathcal{O}_X$-module
of rank $\dim X-\dim Y$
by \ref{log-smoothness}.
Under the isomorphisms
\begin{gather}
\Omega^1_Y(\log \Sigma)
\simeq
\mathcal{O}_Y \frac{dt_1}{t_1} \oplus (\bigoplus_{i=2}^k \mathcal{O}_Y dt_i), \\
\Omega^1_X(\log E)
\simeq
(\bigoplus_{i=1}^l \mathcal{O}_X \frac{dx_i}{x_i})
\oplus
(\bigoplus_{i=l+1}^n \mathcal{O}_X dx_i)
\end{gather}
the morphism
$f^*\Omega_Y^1(\log \Sigma)_x \otimes \mathbb{C}(x)
\to \Omega_X^1(\log E)_x \otimes \mathbb{C}(x)$
is represented by the matrix
\begin{equation}
\label{eq-matrix-representation}
\left(
\begin{array}{ccc|ccc}
a_1 & \dots & a_l & 0 & \dots & 0 \\ \hline
0 & \dots & 0 & & & \\
\vdots & \ddots & \vdots &
\multicolumn{3}{c}{
{\displaystyle
\frac{\partial f_i}{\partial x_j}(0)
}
} \\
0 & \dots & 0 & & &
\end{array}
\right)
\end{equation}
where $i$ and $j$ run through
$2, \dots, k$ and $l+1, \dots, n$ respectively.
The exactness of 
\eqref{eq-exact-seq-logsmooth}
implies
that the matrix \eqref{eq-matrix-representation} is of rank $k$,
and then we may assume
\begin{equation}
\rank
\left(
\frac{\partial f_i}{\partial x_j}(0)
\right)_{2 \le i \le k, l+1 \le j \le l+k-1}
=k-1
\end{equation}
by changing the order of $x_{l+1}, \dots, x_n$.
Replacing $x_{l+1}, \dots, x_{l+k-1}$ by $f_2, \dots, f_k$,
we obtain a new local coordinates
$(x_1, \dots, x_n)$
at $x$,
under which the morphism $f$ is given in the form
\begin{equation}
\label{eq-local-description}
t_1=x_1^{a_1} \cdots x_l^{a_l},
t_2=x_{l+1}, \dots, t_k=x_{l+k-1}
\end{equation}
around $x$. 
We set $f_s \colon X_s \to \Delta=\Delta \times \{s\}$
by the Cartesian square 
\begin{equation*}
\xymatrix{
X_s \ar[r]\ar[d]_-{f_s}&\ar[d] X\ar[d] \\
\Delta \ar[r]& Y
}
\end{equation*} 
for any $s=(t_2, \dots, t_k) \in \Delta^{k-1}$.
Then $X_s$ is smooth,
$f_s$ is smooth over $\Delta^*=\Delta \setminus \{0\}$
and $\Supp f_s^{-1}(0)$ is a simple normal crossing divisor on $X_s$
by the local description \eqref{eq-local-description}.
Hence $R^i(f_s)_*\Omega_{X_s/\Delta}(\log (E \cap X_s))$
and $R^i(f_s)_*\Omega^p_{X_s/\Delta}(\log (E \cap X_s))$
are locally free of finite rank
for every $i, p$ by \cite[(2.18) Theorem]{steenbrink1}
and by \cite[(2.11) Theorem]{steenbrink2}.
Therefore
$R^if_*\Omega_{X/Y}(\log E)$
and $R^if_*\Omega_{X/Y}^p(\log E)$
are locally free $\mathcal{O}_Y$-modules of finite rank
for all $i, p$ by the base change theorem.
Once we know that $R^if_*\Omega_{X/Y}(\log E)$
is locally free,
it is the lower canonical extension
of its restriction to $Y^*=Y \setminus \Sigma$
by \cite[(2.20) Proposition]{steenbrink1}.

Next, we consider the spectral sequence
\begin{equation}
\label{eq-ss-steenbrink}
E_r^{p,q}(Rf_*\Omega_{X/Y}(\log E), F)
\Rightarrow
E^{p+q}(Rf_*\Omega_{X/Y}(\log E))=R^{p+q}f_*\Omega_{X/Y}(\log E)
\end{equation}
and denote the morphism of $E_r$-terms by
\begin{equation*}
d_r^{p,q} \colon E_r^{p,q}(Rf_*\Omega_{X/Y}(\log E), F)
\to E_r^{p+r,q-r+1}(Rf_*\Omega_{X/Y}(\log E), F)
\end{equation*}
for a while.
Then $d_r^{p,q}|_{Y^*}=0$ for all $p,q$ and $r \ge 1$
because the restriction of this spectral sequence to $Y^*$
degenerates at $E_1$-terms.
Since
\begin{equation*}
E_1^{p,q}(Rf_*\Omega_{X/Y}(\log E), F)
\simeq
R^qf_*\Omega^p_{X/Y}(\log E)
\end{equation*}
is a locally free $\mathcal{O}_Y$-module of finite rank
for all $p,q$,
we have $d_1^{p,q}=0$ for all $p,q$ by Lemma \ref{c-lem3.1}.
This implies that
\begin{equation*}
E_2^{p,q}(Rf_*\Omega_{X/Y}(\log E), F)
\simeq
E_1^{p,q}(Rf_*\Omega_{X/Y}(\log E), F)
\end{equation*}
is locally free for all $p,q$
and that $d_2^{p,q}=0$ for all $p,q$ by Lemma \ref{c-lem3.1} again.
Inductively, we obtain $d_r^{p,q}=0$ for all $p,q$ and $r \ge 1$.
Thus the spectral sequence \eqref{eq-ss-steenbrink}
degenerates at $E_1$-terms,
or equivalently,
\eqref{eq-exact-steenbrink} is exact.
\end{proof}

\begin{rem}
\label{c-rem3.4}
In \cite{steenbrink2}, 
$f_s$ is assumed to be a projective morphism.
However, we can check that the proof of (2.11) Theorem
in \cite{steenbrink2}
is also valid to a proper morphism from a K\"ahler manifold
by using results in
\cite[I.2.5 Almost K\"ahler $V$-manifolds]{Peters-SteenbrinkMHS}. 
See also Theorem \ref{c-thm6.9} below. 
\end{rem}

\begin{cor}
\label{c-cor3.5}
In the situation of Lemma \ref{c-lem3.3},
we have the canonical isomorphisms
\begin{align*}
&R^if_*F^p\Omega_{X/Y}(\log E)
\simeq
F^pR^if_*\Omega_{X/Y}(\log E), \\
&R^if_*\Omega^p_{X/Y}(\log E)
\simeq
\Gr_F^pR^if_*\Omega_{X/Y}(\log E)
\end{align*}
for all $i, p$.
In particular,
$F^pR^if_*\Omega_{X/Y}(\log E)$
is a subbundle of
$R^if_*\Omega_{X/Y}(\log E)$.
\end{cor}

\begin{lem}
\label{c-lem3.6}
Let $f \colon X \to Y$ be a proper surjective morphism
between smooth complex varieties.
Assume that there exists a smooth divisor $\Sigma$
such that 
\begin{itemize}
\item
$f$ is smooth over $Y^*=Y \setminus \Sigma$, and 
\item 
$E=(f^*\Sigma)_{\red}$ is a simple normal crossing divisor on $X$
having finitely many irreducible components. 
\end{itemize} 
Then there exists a closed analytic subset
$\Sigma_0 \subset \Sigma$
with $\dim \Sigma_0 \le \dim Y-2$,
such that $\Omega^1_{X/Y}(\log E)$
is locally free
on $f^{-1}(Y \setminus \Sigma_0)$.
\end{lem}
\begin{proof}
We may assume that $\Sigma$ is irreducible.
Let $E=\sum_{i=1}^NE_i$
be the irreducible decomposition of $E$.
For a 
nonempty subset $I \subset \{1, \dots, N\}$,
we set
$E_I=\bigcap_{i \in I}E_i$,
which is a smooth closed subvariety of $X$.
If $f(E_I) \not= \Sigma$,
we set $\Sigma_I=f(E_I)$,
which is a closed analytic subset of $\Sigma$.
If $f(E_I)=\Sigma$, then
there exists a closed analytic subset $\Sigma_I \subsetneq \Sigma$
such that
$f|_{E_I} \colon E_I \to \Sigma$
is smooth over $\Sigma \setminus \Sigma_I$.
We are going to check that the closed analytic subset
\begin{equation*}
\Sigma_0
:=\bigcup_{\emptyset \not= I \subset \{1, \dots, N\}}\Sigma_I
\end{equation*}
satisfies the desired property.
We have $\Sigma_0 \not= \Sigma$, by definition.
Therefore $\dim \Sigma_0 \le \dim Y-2$
because $\Sigma$ is irreducible.
Then, it suffices to prove that
$\Omega^1_{X/Y}(\log E)$
is locally free on $f^{-1}(Y \setminus \Sigma_0)$.
A point $x \in E \cap f^{-1}(Y \setminus \Sigma_0)$
defines a 
nonempty 
subset $I \subset \{1, \dots, N\}$ 
by $I=\{i \mid x \in E_i\}$.
Then $x \in E_I$ and $f(E_I)=\Sigma$.
Take local coordinates $x_1, \dots, x_n$ and $t_1, \dots, t_k$
centered at $x$ and $f(x)$ on $X$ and $Y$ respectively,
satisfying the following conditions:
\begin{itemize}
\item
$\Sigma=\{t_1=0\}$ on $Y$, and
\item
$f^*t_1=x_1^{a_1} \cdots x_l^{a_l}$
for some $a_1, \dots, a_l \in \mathbb{Z}_{>0}$.
\end{itemize}
We set $f_i=f^*t_i$ for $i=2, \dots, k$.
Then $E_I=\{x_1=\dots=x_l=0\}$
and the morphism
$(f|_{E_I})^*\Omega^1_{\Sigma} \to \Omega^1_{E_I}$
is represented by the matrix
\begin{equation*}
\left(
\frac{\partial f_i}{\partial x_j}(0, \dots, 0, x_{l+1}, \dots, x_n)
\right)_{2 \le i \le k, l+1 \le j \le n}
\end{equation*}
via the isomorphisms 
$(f|_{E_I})^*\Omega^1_{\Sigma} \simeq \bigoplus_{j=2}^k\mathcal{O}_{E_I} 
f^*dt_j$ and
$\Omega^1_{E_I} \simeq \bigoplus_{i=l+1}^n\mathcal{O}_{E_I} dx_i$. 
Since $x \in f^{-1}(\Sigma \setminus \Sigma_I)$,
the morphism $f|_{E_I}$ is smooth at $x$.
Then
\begin{equation*}
\rank
\left(
\frac{\partial f_i}{\partial x_j}(0)
\right)_{2 \le i \le k, l+1 \le j \le n}
=k-1,
\end{equation*}
which implies 
that the matrix \eqref{eq-matrix-representation}
in the proof of Lemma \ref{c-lem3.3}
is of rank $k$. 
Therefore the canonical morphism
$f^*\Omega^1_Y(\log \Sigma)_x \otimes \mathbb{C}(x)
\to \Omega^1_X(\log E)_x \otimes \mathbb{C}(x)$
is injective, 
by which we conclude
that $\Omega^1_{X/Y}(\log E)$ is locally free around $x$.
\end{proof}

\begin{say}
\label{c-say3.7}
We return to the situation in \ref{para:3}.
For the moment,
we assume that
there exist another semisimplicial variety
$Z_{\bullet}$
and a morphism of semisimplicial varieties
$\sigma \colon Z_{\bullet} \to (X, D)_{\bullet}$
satisfying the conditions
\begin{itemize}
\item 
$Z_n$ is smooth and K\"ahler, 
\item
$\sigma_n \colon Z_n \to (X, D)_n$
is a projective surjective morphism,
\item
for
$g_n:=f_n \sigma_n=f\varepsilon_n \sigma_n \colon Z_n \to Y$,
the divisor $E_n:=(g_n^*\Sigma)_{\red}$
is a simple normal crossing divisor on $Z_n$
having finitely many irreducible components, and
\item
$\sigma_n \colon Z_n \to (X, D)_n$
is isomorphic over $Y^*$
\end{itemize}
for every $n \in \mathbb{Z}_{\ge 0}$.
We obtain an augmentation $\eta \colon Z_{\bullet} \to X$
by setting $\eta=\varepsilon \sigma$. 
The relative log de Rham complex of $Z_n$ over $Y$
is denoted by $\Omega_{Z_n/Y}(\log E_n)$.
Then $\{\Omega_{Z_n/Y}(\log E_n)\}_{n \in \mathbb{Z}_{\ge 0}}$
forms a complex on the semisimplicial variety $Z_{\bullet}$.

For an augmentation of a semisimplicial variety,
we can define the direct image functor
as in \cite[4.1, 4.2]{fujino-fujisawa}
(for the detail, see e.g.~\cite[5.1, 5.2]{DeligneIII},
\cite[5.1.2]{Peters-SteenbrinkMHS}).
The complex $R\varepsilon_*\Omega_{(X, D)_{\bullet}}$
is isomorphic to
$\varepsilon_*\Omega_{(X, D)_{\bullet}}$ defined
in the proof of Theorem \ref{z-thm1.1} ({i})
in the derived category
because $\varepsilon_n \colon (X, D)_n \to X$ is a finite morphism
for all $n$.
On the other hand,
we obtain a complex
$R\eta_*\Omega_{Z_{\bullet}/Y}(\log E_{\bullet})$ on $X$.
Here, we briefly recall the definitions of this complex,
of the finite increasing filtration $L$, and
of the finite decreasing filtration $F$ on it.
First, the complex $R\eta_*\Omega_{Z_{\bullet}/Y}(\log E_{\bullet})$
is given as the total single complex
associated to the double complex 
\begin{equation*}
\xymatrix{
&\vdots \ar[d]& \vdots\ar[d] \\
\cdots \ar[r]& (R(\eta_p)_*\Omega_{Z_p/Y}(\log E_p))^q
\ar[r]^-{\delta}\ar[d]_-{(-1)^pd}& 
(R(\eta_{p+1})_*\Omega_{Z_{p+1}/Y}(\log E_{p+1}))^q \ar[r]
\ar[d]^-{(-1)^{p+1}d}& \cdots \\
\cdots \ar[r]& (R(\eta_p)_*\Omega_{Z_p/Y}(\log E_p))^{q+1}
\ar[r]^-{\delta}\ar[d]&
(R(\eta_{p+1})_*\Omega_{Z_{p+1}/Y}(\log E_{p+1}))^{q+1}\ar[r]
\ar[d]&\cdots \\
& \vdots & \vdots
}
\end{equation*} 
that is,
\begin{equation*}
(R\eta_*\Omega_{Z_{\bullet}/Y}(\log E_{\bullet}))^n
=\bigoplus_p(R(\eta_p)_*\Omega_{Z_p/Y}(\log E_p))^{n-p},
\end{equation*}
where $R(\eta_p)_*\Omega_{Z_p/Y}(\log E_p)$ is regarded
as {\itshape a genuine complex} on $X$
by using the Godement resolutions
(cf.~\cite[4.1]{fujino-fujisawa}).
The filtrations $L$ and $F$ are defined by
\begin{gather*}
L_m(R\eta_*\Omega_{Z_{\bullet}/Y}(\log E_{\bullet}))^n
=\bigoplus_{p \ge -m}(R(\eta_p)_*\Omega_{Z_p/Y}(\log E_p))^{n-p}, \\
F^r(R\eta_*\Omega_{Z_{\bullet}/Y}(\log E_{\bullet}))^n
=\bigoplus_pF^r(R(\eta_p)_*\Omega_{Z_p/Y}(\log E_p))^{n-p}
\end{gather*}
for all $m,n,r$.
Therefore we have
\begin{equation}
\label{eq-GrL-OmegaZ}
(\Gr_m^LR\eta_*\Omega_{Z_{\bullet}/Y}(\log E_{\bullet}), F)
\simeq
(R(\eta_{-m})_*\Omega_{Z_{-m}/Y}(\log E_{-m})[m], F)
\end{equation}
in the derived category. 
From the morphism
$\sigma \colon Z_{\bullet} \to (X, D)_{\bullet}$,
we obtain a morphism of bifiltered complexes
\begin{equation}
\label{eq:10}
(\varepsilon_*\Omega_{(X, D)_{\bullet}/Y}, L, F)
\to
(R\eta_*\Omega_{Z_{\bullet}/Y}(\log E_{\bullet}), L, F),
\end{equation}
which induces a morphism
\begin{equation}
\label{eq-GrmL-epsilo-eta}
\begin{split}
\Gr_m^L\Gr_F^0\varepsilon_*\Omega_{(X, D)_{\bullet}/Y}
&\simeq
(\varepsilon_{-m})_*\mathcal{O}_{(X, D)_{-m}} \\
&\to
R(\eta_{-m})_*\mathcal{O}_{Z_{-m}}
\simeq
\Gr_m^L\Gr_F^0R\eta_*\Omega_{Z_{\bullet}/Y}(\log E_{\bullet})
\end{split}
\end{equation}
for all $m$.
Because $\sigma_n$ induces the isomorphism
$\mathcal{O}_{(X, D)_n} \xrightarrow{\simeq} R(\sigma_n)_*\mathcal{O}_{Z_n}$
for all $n$,
we have the isomorphisms
\begin{equation*}
(\varepsilon_{-m})_*\mathcal{O}_{(X, D)_{-m}}
\simeq
R(\varepsilon_{-m})_*\mathcal{O}_{(X, D)_{-m}}
\simeq
R(\varepsilon_{-m})_*R(\sigma_{-m})_*\mathcal{O}_{Z_{-m}}
\simeq
R(\eta_{-m})_*\mathcal{O}_{Z_{-m}}
\end{equation*}
for all $m$.
Therefore the morphism \eqref{eq-GrmL-epsilo-eta}
is an isomorphism for all $m$ in the derived category,
which implies
\begin{equation}
\label{eq-GrF0-OmegaZ}
(\Gr_F^0\varepsilon_*\Omega_{(X, D)_{\bullet}/Y}, L)
\simeq
(\Gr_F^0R\eta_*\Omega_{Z_{\bullet}/Y}(\log E_{\bullet}), L)
\end{equation}
in the filtered derived category.
\end{say}

Now, we complete the proof of Theorem \ref{z-thm1.1}.

\begin{proof}
[{Proof of Theorem \ref{z-thm1.1} {\em{(ii)--(iv)}}}] 
First, we prove (ii).
The uniqueness of the filtration $F$ on ${}^l\mathcal{V}^k_{Y^*}$
follows from \cite[Corollary 5.2]{fujino-fujisawa}.
Therefore we may work locally on $Y$.
Then after shrinking $Y$ to a relatively compact open subset,
we can take $Z_{\bullet}$
and $\sigma_{\bullet} \colon Z_{\bullet} \to (X, D)_{\bullet}$
in \ref{c-say3.7}
by the theorem of resolution of singularities (see 
\cite[Section 13]{bierstone-milman}).
By Lemma \ref{c-lem3.6},
there exists a closed analytic subset $\Sigma_0 \subset \Sigma$
with $\dim \Sigma_0 \le \dim Y-2$
such that $\Sigma \setminus \Sigma_0$
is a smooth divisor in $Y \setminus \Sigma_0$,
and that $\Omega^1_{Z_n/Y}(\log E_n)$
is locally free over $g_n^{-1}(Y \setminus \Sigma_0)$
for all $n \in \mathbb{Z}_{\ge 0}$.
By setting $Y_0:=Y \setminus \Sigma_0$,
we trivially have $Y^* \subset Y_0 \subset Y$.

Now we set
\begin{equation*}
K(\log)_{\mathcal{O}}=Rf_*R\eta_*\Omega_{Z_{\bullet}/Y}(\log E_{\bullet})
\end{equation*}
equipped with the induced filtrations $L$ and $F$.
Because $\sigma$ is isomorphic over $Y^*$,
the morphism \eqref{eq:10} induces an isomorphism
\begin{equation}
\label{eq-isom-over-Y*}
(K_{\mathcal{O}}, L, F)|_{Y^*}
\xrightarrow{\simeq}
(K(\log)_{\mathcal{O}}, L, F)|_{Y^*},
\end{equation}
and then a morphism of complexes
$\beta \colon K_{\mathbb{R}} \to K(\log)_{\mathcal{O}}|_{Y^*}$
is defined by
\begin{equation}
\beta \colon K_{\mathbb{R}} \xrightarrow{\alpha} K_{\mathcal{O}}|_{Y^*}
\xrightarrow{\simeq} K(\log)_{\mathcal{O}}|_{Y^*},
\end{equation}
where $K_{\mathbb{R}}$ and $\alpha$
are given in \eqref{eq:1}.
A triple
\begin{equation}
K(\log)=
((K_{\mathbb{R}},L), (K(\log)_{\mathcal{O}},L,F), \beta)
\end{equation}
satisfies the conditions
\ref{item:7}--\ref{item:9}
because $K|_{Y^*} \simeq K(\log)|_{Y^*}$ as above.
Then the triple on $Y_0$
\begin{equation}
K(\log)|_{Y_0}=((K_{\mathbb{R}},L), (K(\log)_{\mathcal{O}},L,F)|_{Y_0}, \beta)
\end{equation}
satisfies all the assumptions in Theorem \ref{lem:4}
by \eqref{eq-GrL-OmegaZ} and Lemma \ref{c-lem3.3}.
Applying Theorem \ref{lem:4} to $K(\log)|_{Y_0}$,
we conclude
\begin{gather}
H^n(K(\log)_{\mathcal{O}})|_{Y_0}
\simeq ({}^\ell H^n(K(\log)_{\mathcal{O}})|_{Y^*})|_{Y_0}
\simeq ({}^\ell H^n(K_{\mathcal{O}})|_{Y^*})|_{Y_0} \\
L_mH^n(K(\log)_{\mathcal{O}})|_{Y_0}
\simeq ({}^\ell L_mH^n(K(\log)_{\mathcal{O}})|_{Y^*})|_{Y_0}
\simeq ({}^\ell L_mH^n(K_{\mathcal{O}})|_{Y^*})|_{Y_0}
\end{gather}
from \eqref{eq-isom-over-Y*},
and that
$\Gr_F^a\Gr_m^LH^n(K(\log)_{\mathcal{O}})|_{Y_0}$
is locally free of finite rank for every $a,m,n$.
By the isomorphism
$\mathcal{V}^k_{Y^*} \simeq H^k(K_{\mathcal{O}})|_{Y^*}$
as in the proof of Theorem \ref{z-thm1.1} (i) above,
we obtain a filtration $F$
on $({}^\ell \mathcal{V}^k_{Y^*})|_{Y_0}$
satisfying the two conditions in Theorem \ref{z-thm1.1} (ii) on $Y_0$.
Then Lemma 1.11.2 in \cite{kashiwara}
together with Schmid's nilpotent orbit theorem 
(see \cite[(4.12)]{schmid}) 
for each $\Gr_m^L\mathcal{V}^k_{Y^*}$
implies the conclusion of Theorem \ref{z-thm1.1} (ii)
on the whole $Y$.

Next, we will prove (iii).
By \ref{item:3} of Theorem \ref{lem:4} for $K(\log)|_{Y_0}$,
we have
\begin{equation}
H^k(\Gr_F^aK(\log)_{\mathcal{O}})|_{Y_0}
\simeq
\Gr_F^aH^k(K(\log)_{\mathcal{O}})|_{Y_0}
\simeq
\Gr_F^a({}^\ell \mathcal{V}^k_{Y^*})|_{Y_0}
\end{equation}
for every $a,k$.
On the other hand,
\begin{equation}
\begin{split}
\Gr_F^0K(\log)_{\mathcal{O}}
&=\Gr_F^0Rf_*R\eta_*\Omega_{Z_{\bullet}/Y}(\log E_{\bullet}) \\
&\simeq
\Gr_F^0Rf_*\varepsilon_*\Omega_{(X, D)_{\bullet}/Y}
=\Gr_F^0K_{\mathcal{O}}
\simeq Rf_*\mathcal{O}_X(-D)
\end{split}
\end{equation}
by \eqref{eq-GrF0-OmegaZ} and \ref{qis-O(-D)}.
Thus we obtain an isomorphism
\begin{equation}
R^{d-i}f_*\mathcal{O}_X(-D)|_{Y_0}
\simeq
H^{d-i}(\Gr_F^0K(\log)_{\mathcal{O}})|_{Y_0}
\simeq
\Gr_F^0({}^\ell \mathcal{V}^{d-i}_{Y^*})|_{Y_0}
\end{equation}
for every $i$.
Because $R^{d-i}f_*\mathcal{O}_X(-D)$
is locally free of finite rank by Theorem \ref{z-thm1.4},
and because $\Sigma_0=Y \setminus Y_0$ has the codimension at least two,
the isomorphism above
extends uniquely to the whole $Y$.

By Grothendieck duality (see \cite{rrv}),
we obtain (iv) from (iii). 
\end{proof}

\begin{rem}
\label{rem:2} 
As already mentioned in Remark \ref{z-rem1.2}, the local system 
$R^k(f|_{X^*\setminus D^*})_!\mathbb R_{X^* \setminus D^*}$
underlies {\itshape an admissible} graded polarizable variation
of $\mathbb{R}$-mixed Hodge structure on $Y^*$. 
(For the definition of admissibility,
see \cite[(3.13) Properties]{sz},
\cite[1.8, 1.9]{kashiwara},
\cite[Definition 14.49]{Peters-SteenbrinkMHS},
\cite[Definition 3.11]{fujino-fujisawa}, and so on.) 
In order to check the admissibility,
we may assume that $(Y, \Sigma)=(\Delta, \{0\})$ from the beginning.
We may further assume that all the monodromy automorphisms of
$R^i(f_S)_\ast\mathbb{R}_S|_{\Delta^*}$ around the origin
are unipotent for all strata $S$ and for all $i \in \mathbb{Z}$
by \cite[Lemma 1.9.1]{kashiwara}.
Then the extension of the Hodge filtration
has been already given by Theorem \ref{z-thm1.1} (ii).
On the other hand,
the existence of the relative monodromy weight filtration
can be proved by the same way
as the proof of Lemma 4.10 of \cite{fujino-fujisawa}.
Here we remark that
the coincidence of the monodromy weight filtration
and the weight filtration
on the limit mixed Hodge structure
in \cite[(5.9) Theorem]{steenbrink1}
holds true for a proper surjective morphism
$f \colon X \to \Delta$ from a K\"ahler manifold $X$ to the unit disc $\Delta$
such that $f$ is smooth over $\Delta^*$,
that the local system $R^if_*\mathbb{R}|_{\Delta^*}$
is of unipotent monodromy for every $i$,
and that $f^{-1}(0)_{\red}$ is a simple normal crossing divisor
(cf. \cite[4.2.5 Remarque]{saito1},
\cite[(5.2) Th\'eor\`eme]{Guillen-NavarroAznarCI}).
\end{rem}

The following theorem 
is an easy consequence of the 
proof of Theorem \ref{z-thm1.4}. 
We will use it in the proof of Theorem \ref{z-thm1.5}. 

\begin{thm}\label{c-thm3.8} 
In Theorem \ref{z-thm1.1}, 
for every $i$, 
there exists a finite filtration of locally free sheaves 
\begin{equation*}
0=\mathcal E^i_0 \subset \mathcal E^i_1\subset \cdots 
\subset \mathcal E^i_{l_i}=R^if_*\omega_{X/Y}(D) 
\end{equation*} 
such that 
\begin{equation*}
\mathcal E^i_{j+1} /\mathcal E^i_j
\end{equation*} 
is isomorphic to a direct summand of 
\begin{equation*}
\bigoplus _{\mathrm{finite}} R^\alpha f_*\omega_{S_\beta/Y}, 
\end{equation*} 
where $\alpha$ is a nonnegative integer and $S_\beta$ is a 
stratum of $(X, D)$, for every $j$. 

Moreover, if $\pi\colon Y\to Z$ is a projective bimeromorphic 
morphism of complex varieties, then 
\begin{equation*}
R^p\pi_*R^if_*\omega_X(D)=0
\end{equation*} 
holds for every $p>0$. In particular, we have 
\begin{equation*}
\pi_*R^if_*\omega_X(D)\simeq R^i(\pi\circ f)_*\omega_X(D). 
\end{equation*} 
\end{thm}

\begin{proof}
By Theorem \ref{z-thm1.4}, there exists a finite filtration of 
locally free sheaves 
\begin{equation*}
0= \mathcal F^{d-i}_0 \subset 
\mathcal F^{d-i}_1\subset 
\cdots \subset \mathcal F^{d-i}_{l_i}=R^{d-i}f_*\mathcal O_X(-D)
\end{equation*} 
such that 
\begin{equation*}
\mathcal F^{d-i}_{j+1}/\mathcal F^{d-i}_j
\end{equation*} 
is isomorphic to a direct summand of 
\begin{equation*}
\bigoplus _{\mathrm{finite}} R^{r}f_*\mathcal O_{S_\beta}, 
\end{equation*} 
where $S_\beta$ is a stratum of $(X, D)$ 
and $r$ is a nonnegative integer, for every $j$. 
We put 
\begin{equation*}
\mathcal E^i_j:=\mathcal {H}om _{\mathcal O_Y}
(\mathcal F^{d-i}_{l_i}/\mathcal F^{d-i}_{l_i-j}, 
\mathcal O_Y)
\end{equation*} 
for every $j$. 
Then, by Grothendieck duality (see \cite{rrv}), 
we obtain a desired filtration of $R^if_*\omega_{X/Y}(D)$. 
By Theorem \ref{z-thm2.11}, 
we have $R^p\pi_*R^\alpha f_*\omega_{S_\beta}=0$ for every $p>0$. 
This implies that 
$R^p\pi_*\left(\mathcal E^i_j\otimes \omega_Y\right)=0$ holds 
for every $p>0$ and every $j$. 
Thus we obtain $R^p\pi_*R^if_*\omega_X(D)=0$ 
for every $p>0$. 
Hence we have $\pi_*R^if_*\omega_X(D)\simeq 
R^i(\pi\circ f)_*\omega_X(D)$. 
We finish the proof. 
\end{proof}

We close this section with the proof of 
Theorem \ref{z-thm1.3}, which is a generalization 
of the Fujita--Zucker--Kawamata semipositivity theorem 
(see, for example, 
\cite[Section 5]{fujino-fujisawa}, 
\cite[Corollary 2]{fujino-fujisawa-saito}, 
\cite{fujino-fujisawa2}, and so on). 

\begin{proof}[Proof of Theorem \ref{z-thm1.3}] 
Let 
\[
\varphi_C\colon C\longrightarrow V\overset{\varphi}{\longrightarrow} X
\] 
be a morphism from a smooth projective curve $C$. 
Then, by Theorem \ref{z-thm1.1} (iv), the proof of \cite[Theorem 5.21]{fujino-fujisawa} 
works for $\varphi_C\colon C\to X$. 
Hence, we obtain the desired semipositivity of $\varphi^*R^if_*\omega_{X/Y}(D)$. 
We finish the proof. 
\end{proof}

We note that Theorems \ref{z-thm1.1} and 
\ref{z-thm1.3} have already played a crucial 
role when $f\colon (X, D)\to Y$ is 
algebraic. We recommend that the interested 
reader looks at \cite{fujino-moduli}, 
\cite{fujino-basic-slc}, \cite{fujino-hyperbolic}, 
\cite{fujino-on-quasi}, \cite{fujino-fujisawa-liu}, 
\cite{fujino-hashizume}, and so on. 

\section{Proof of Theorem \ref{z-thm1.5}}
\label{z-sec5}

In this section, we will prove Theorem \ref{z-thm1.5} 
by using Theorem \ref{c-thm3.8}. 
In Section \ref{z-sec6}, we will see that Theorem 
\ref{z-thm1.7} follows from Theorem \ref{z-thm1.5}. 

\begin{proof}[Proof of Theorem \ref{z-thm1.5}]
In Step \ref{z-step1.5.1} and Step \ref{z-step1.5.2}, we will prove (i) and (ii), respectively. 
\setcounter{step}{0}
\begin{step}\label{z-step1.5.1}
In this step, we will prove (i). 

We take an arbitrary point $P\in Y$. 
It is sufficient to prove (i) around $P$. 
By Lemma \ref{z-lem2.9}, 
we may assume that $(X, D)$ is an analytic globally 
embedded simple normal crossing pair 
and that there exists the following commutative diagram: 
\begin{equation*}
\xymatrix{
X\ar[d]_-f\ar@{^{(}->}[r]& M\ar[d]^-{q_M}\\
Y\ar@{^{(}->}[r]_-{\iota_Y} & \Delta^m, 
}
\end{equation*}
where $M$ is the ambient space of $(X, D)$, 
such that $q_M$ is projective and $\iota_Y(P)=0\in \Delta^m$. 
By taking a suitable resolution of singularities of 
$Y$ (see \cite[Sections 12 and 13]{bierstone-milman}), 
there exist a projective bimeromorphic 
morphism $\psi\colon Y'\to Y$ from a 
smooth complex variety $Y'$ and a simple normal 
crossing divisor $\Sigma'$ on $Y'$ such that 
every stratum of $(X, D)$ is smooth over $Y
\setminus \psi(\Sigma')$. 
Then, by taking a suitable resolution of singularities 
of $M$ (see \cite[Sections 12 and 13]{bierstone-milman}) and 
applying Lemma \ref{z-lem2.8}, 
we may assume that 
\begin{equation*}
f'\colon X\overset{f}{\longrightarrow} 
Y \overset{\psi^{-1}}{\dashrightarrow} Y'
\end{equation*}
is a 
projective morphism. 
Hence we have the following commutative diagram: 
\begin{equation*}
\xymatrix{
X\ar[d]_-{f'} \ar@{=}[r]& X\ar[d]^-f\\
Y'\ar[r]_-{\psi} & Y 
}
\end{equation*} 
such that 
every stratum of $(X, D)$ is smooth over 
$Y'\setminus \Sigma'$. 
By Theorem \ref{c-thm3.8}, $R^qf'_*\omega_{X/Y'}(D)$ is locally
free and has a finite filtration as in Theorem \ref{c-thm3.8}. 
Since $\psi$ is 
projective bimeromorphic and 
$R^qf_*\omega_X(D)\simeq \psi_* R^qf'_*\omega_X(D)$ by Theorem 
\ref{c-thm3.8}, $R^qf_*\omega_X(D)$ is torsion-free. 
This is what we wanted. 
\end{step} 
\begin{step}\label{z-step1.5.2}
In this step, we will prove (ii). 

We take an arbitrary point $P\in Z$. 
It is sufficient to prove (ii) around $P$. 
As in Step \ref{z-step1.5.1}, after shrinking $Z$ suitably, 
by Lemma \ref{z-lem2.9}, a suitable 
resolution of singularities (see \cite[Sections 12 and 13]{bierstone-milman}), 
and Lemma \ref{z-lem2.8}, 
we may assume that 
there exists the following commutative diagram: 
\begin{equation*}
\xymatrix{
X\ar@{=}[r]\ar[d]_-{f'}&X\ar[d]_-f\ar@{^{(}->}[r]& M\ar[dd]^-{q_M}\\
Y' \ar[r]_-\psi& Y\ar[d]_-\pi& &\\ 
&Z \ar@{^{(}->}[r]_-{\iota_Z}  & \Delta^m 
}
\end{equation*}
such that $\iota_Z(P)=0\in \Delta^m$. 
By Theorem \ref{c-thm3.8}, we have 
\begin{equation*}
R^p\pi_*\left(\mathcal A\otimes R^qf_*\omega_X(D)\right)
\simeq R^p\pi_*\left(\mathcal A\otimes \psi_* 
R^qf'_*\omega_X(D)\right)
\simeq R^p(\pi\circ \psi)_*\left(\psi^*
\mathcal A\otimes R^qf'_*\omega_X(D)\right). 
\end{equation*} 
By Theorem \ref{c-thm3.8} again, 
$R^qf'_*\omega_{X/{Y'}}(D)$ has a finite 
filtration as in Theorem \ref{c-thm3.8}. 
Thus we can reduce the problem to the case where 
$X$ is smooth and $D=0$. 
Since $\psi^*\mathcal A$ is $(\pi\circ \psi)$-nef and 
$(\pi\circ \psi)$-big over $Z$, 
we get the desired vanishing theorem by Theorem 
\ref{z-thm2.11}. We finish the proof of (ii). 
\end{step}
We finish the proof of Theorem \ref{z-thm1.5}. 
\end{proof}

\begin{rem}\label{z-rem5.1} 
By the above proof, we see that 
Theorem \ref{z-thm1.5} (ii) holds under a weaker assumption 
that $\mathcal A$ is $\pi$-nef and $\pi$-big over $Z$ (see 
Theorem \ref{z-thm2.11}). 
\end{rem}

\section{Proof of Theorem 
\ref{z-thm1.7}}\label{z-sec6}

In this section, we will prove Theorem \ref{z-thm1.7} 
by using Theorem \ref{z-thm1.5}. As we mentioned before, 
Theorem \ref{z-thm1.7} (iii) is an 
easy consequence of Theorem \ref{z-thm1.7} (i) and 
(ii). 

\begin{proof}[Proof of Theorem \ref{z-thm1.7}]
In Step \ref{z-step1.7.1}, we will prove Theorem 
\ref{z-thm1.7} (i). Then, in Steps \ref{z-step1.7.2} and 
\ref{z-step1.7.3}, 
we will prove Theorem \ref{z-thm1.7} (ii) and (iii), respectively. 
\setcounter{step}{0}
\begin{step}\label{z-step1.7.1} 
In this step, we will prove Theorem \ref{z-thm1.7} (i). 

By replacing $Y$ with $f(X)$, we may assume that 
$f(X)=Y$. 
Let $P\in Y$ be an arbitrary point. 
It is sufficient to prove the statement after shrinking 
$Y$ around $P$ suitably. 
By Lemma \ref{z-lem2.9}, we may 
assume that $(X, D)$ is an analytic globally embedded 
simple normal crossing pair and that 
there exists the following 
commutative diagram: 
\begin{equation*}
\xymatrix{
X\ar[d]_-f\ar@{^{(}->}[r]& M\ar[d]^-{q_M}\\
Y\ar@{^{(}->}[r]_-{\iota_Y} & \Delta^m, 
}
\end{equation*}
where $M$ is the ambient space of $(X, D)$, 
such that $q_M$ is projective and $\iota_Y(P)=0\in 
\Delta^m$. 
By using Lemma \ref{z-lem2.10} 
finitely many times, we can decompose $X=X'+X''$ as 
follows:~$X'$ is the 
union of all strata of $(X, D)$ that are not mapped onto 
irreducible components of $Y=f(X)$ and $X''=X-X'$. 
We put 
\begin{equation*}
K_{X'}+D_{X'}:=(K_X+D)|_{X'}
\end{equation*} 
and 
\begin{equation*}
K_{X''}+D_{X''}:=(K_X+D)|_{X''}-X'|_{X''}. 
\end{equation*} 
We note that 
$(X'', D_{X''})$ is an analytic globally embedded simple 
normal crossing pair such that $D_{X''}$ is 
reduced and that every stratum of $(X'', D_{X''})$ is 
mapped onto some irreducible component of $Y$. 
We consider the following short exact sequence: 
\begin{equation*}
0\to \mathcal O_{X''}(K_{X''}+D_{X''})
\to \mathcal O_X(K_X+D)\to 
\mathcal O_{X'}(K_{X'}+D_{X'})\to 0. 
\end{equation*} 
By Theorem \ref{z-thm1.5} (i), 
every associated subvariety 
of $R^qf_*\mathcal O_{X''}(K_{X''}+D_{X''})$ is an 
irreducible component of $Y$ for every $q$. 
Note that every associated subvariety 
of $R^qf_*\mathcal O_{X'}(K_{X'}+D_{X'})$ is contained 
in $f(X')$ for every $q$. 
Thus, the connecting homomorphisms 
\begin{equation*}
\delta\colon R^qf_*\mathcal O_{X'}
(K_{X'}+D_{X'})\to R^{q+1}f_*\mathcal O_{X''}(K_{X''}
+D_{X''})
\end{equation*} 
are zero for all $q$. 
Hence we obtain the following short exact sequence 
\begin{equation}\label{z-eq6.1}
0\to R^qf_*\mathcal O_{X''} (K_{X''}+D_{X''})
\to R^qf_*\mathcal O_X(K_X+D)\to R^qf_*\mathcal O_{X'}(K_{X'}+D_{X'})
\to 0
\end{equation} 
for every $q$. 
By induction on $\dim f(X)$, 
every associated subvariety of $R^qf_*\mathcal O_{X'}(K_{X'}+D_{X'})$ is 
the $f$-image of some stratum of $(X', D_{X'})$ for 
every $q$. 
Therefore, every associated subvariety 
of $R^qf_*\mathcal O_X(K_X+D)$ is the $f$-image 
of some stratum of $(X, D)$ for 
every $q$ by \eqref{z-eq6.1}. 
\end{step}
\begin{step}\label{z-step1.7.2}
In this step, we will prove Theorem \ref{z-thm1.7} (ii). 

We may assume that $f(X)=Y$ and $\pi\circ f(X)=Z$. 
Let $P\in Z$ be an arbitrary point. 
It is sufficient to prove the desired vanishing theorem 
after shrinking $Z$ around $P$ suitably. 
As in Step \ref{z-step1.7.1}, by Lemma \ref{z-lem2.9}, 
we have the following commutative diagram: 
\begin{equation*}
\xymatrix{
X\ar[d]_-{\pi\circ f}\ar@{^{(}->}[r]& M\ar[d]^-{q_M}\\
Z\ar@{^{(}->}[r]_-{\iota_Z} & \Delta^m, 
}
\end{equation*}
where $M$ is the ambient space of $(X, D)$, 
such that $q_M$ is projective and $\iota_Z(P)=0\in 
\Delta^m$. 
By the same argument as in Step \ref{z-step1.7.1}, 
we obtain 
\begin{equation}\label{z-eq6.2}
0\to R^qf_*\mathcal O_{X''} (K_{X''}+D_{X''})
\to R^qf_*\mathcal O_X(K_X+D)\to R^qf_*\mathcal O_{X'}(K_{X'}+D_{X'})
\to 0
\end{equation} 
for every $q$. 
By applying Theorem \ref{z-thm1.5} (ii) to 
every connected component of $X''$, we see that  
\begin{equation*}
R^p\pi_*\left(\mathcal A\otimes 
R^qf_*\mathcal O_{X''}(K_{X''}+D_{X''})\right)=0
\end{equation*} 
holds 
for every $p>0$. 
By induction on $\dim f(X)$, 
we obtain 
\begin{equation*}
R^p\pi_*\left(\mathcal A\otimes 
R^qf_*\mathcal O_{X'}(K_{X'}+D_{X'})\right)=0
\end{equation*} for 
every $p>0$. 
This implies 
\begin{equation*}
R^p\pi_*\left(\mathcal A\otimes 
R^qf_*\mathcal O_X(K_X+D)\right)=0
\end{equation*} 
for every $p>0$. 
This is what we wanted. 
\end{step}
\begin{step}\label{z-step1.7.3} 
In this step, we will prove Theorem \ref{z-thm1.7} (iii). 

It is well known that the injectivity theorem is an easy 
consequence of the strict support condition (see (i)) and 
the vanishing theorem (see (ii)). 
Since we have already proved the strict support 
condition (see (i)) and the vanishing theorem 
(see (ii)) in Steps \ref{z-step1.7.1} and \ref{z-step1.7.2}, 
respectively, we obtain the desired injectivity in (iii). 
For the details, see the proof of \cite[Theorem 3.1 (iii)]
{fujino-analytic-vanishing}. 
\end{step} 
We finish the proof of Theorem \ref{z-thm1.7}. 
\end{proof}

\begin{rem}\label{z-rem6.1} 
Theorem \ref{z-thm1.7} (ii) holds under a weaker assumption that 
$\mathcal A$ is nef and log big over $Z$ with 
respect to $f\colon (X, D)\to Y$. 
We can easily check it by the above proof of Theorem 
\ref{z-thm1.7} (ii) and Remark \ref{z-rem5.1}. 
We do not discuss the details here because we have 
already known a more general statement, that is, 
the vanishing theorem of Reid--Fukuda type 
(see Theorem \ref{z-thm1.9}). 
We note that Theorem \ref{z-thm1.7} (ii) is 
a very special case of Theorem \ref{z-thm1.9} and 
that Theorem \ref{z-thm1.9} follows from 
Theorem \ref{z-thm1.7} (see \cite{fujino-analytic-vanishing}).  
\end{rem}

\section{Supplement to \cite{steenbrink2}}\label{c-sec6}

In this section,
we give a remark on the construction
of the cohomological $\mathbb{Q}$-mixed Hodge complex
$((A_{\mathbb{Q}}, W), (A_{\mathbb{C}}, W, F))$
in \cite[p.536]{steenbrink2}.
More precisely, we will present
a new construction of $(A_{\mathbb{Q}}, W)$ here.
In the context of log geometry,
such a construction 
originates from \cite{Steenbrink3}
and is used in other articles
(e.g.~\cite{Fujisawa-Nakayama}, \cite{Fujisawa2} and so on).
For the case of a semistable reduction,
a new construction of $(A_{\mathbb{Q}}, W)$,
which is similar to \cite{Steenbrink3},
is given in \cite[11.2.6 The Rational Structure]{Peters-SteenbrinkMHS}.
(For the case of a semistable morphism over the polydisc,
see e.g.~\cite{Fujisawa1}.)
Here we will see that the construction
in \cite{Fujisawa2}
works in the situation of \cite{steenbrink2}.

\begin{say}
\label{c-say6.1}
Let $f \colon X \to \Delta$ be a proper surjective morphism
from a smooth complex variety $X$ to the unit disc $\Delta$
satisfying the conditions
\begin{itemize}
\item
$f$ is smooth over $\Delta^*=\Delta \setminus \{0\}$, and
\item
$\Supp f^{-1}(0)$ is a simple normal crossing divisor on $X$
\end{itemize}
as in \cite[(2.1) Notations]{steenbrink2}.
Note that $f^{-1}(0)$ is {\itshape not} assumed to be reduced.
We fix $N \in \mathbb{Z}_{>0}$,
which is a multiple of
all the multiplicities of the irreducible components of $\Supp f^{-1}(0)$ 
in $f^{-1}(0)$,
and consider the morphism
$\sigma \colon \Delta \to \Delta$ given by $\sigma(t)=t^N$.
We define $\widetilde{X}, \pi$  and $\widetilde{f}$
by the commutative diagram
\begin{equation}
\xymatrix{
\widetilde{X}
\ar[rd]^-{\nu} \ar@/^8pt/[rrd]^{\pi} \ar@/_5pt/[rdd]_{\widetilde{f}}
& & \\
& X \times_{\Delta} \Delta \ar[r] \ar[d] & X \ar[d]^-{f}\\
& \Delta \ar[r]_-{\sigma} & \Delta
}
\end{equation}
where $\nu$ is the normalization.
We set $E=\Supp \widetilde{f}^{-1}(0)$,
which is an effective Cartier divisor on $\widetilde{X}$.
The irreducible decomposition of $E$
is written in
$E=\bigcup_{i=1}^lE_i$.
The closed immersion $E_i \hookrightarrow \widetilde{X}$
is denoted by $a_i$.
\end{say}

\begin{say}
\label{c-say6.2}
We recall the local description of $\widetilde{X}$ and $\widetilde{f}$
given in the proof of \cite[(2.2) Lemma]{steenbrink2}.
For any point of $\widetilde{X}$,
there exist an open neighborhood $\widetilde{U}$ in $\widetilde{X}$,
$d_1, \dots, d_k \in \mathbb{Z}_{>0}$
with $\gcd(d_1, \dots, d_k)=1$,
and $e \in \mathbb{Z}_{>0} \cap (\bigcap_{i=1}^kd_i\mathbb{Z})$
with $N \in e\mathbb{Z}$
such that $\widetilde{U}$ and $\widetilde{f}|_{\widetilde{U}}$
are described
by using $d_1, \dots, d_k, e$
as follows.
By setting $c_i:=e/d_i \in \mathbb{Z}_{>0}$
and $G:=\bigoplus_{i=1}^k\mathbb{Z}/c_i\mathbb{Z}$,
the kernel of the morphism 
\begin{equation}
G=\bigoplus_{i=1}^k\mathbb{Z}/c_i\mathbb{Z}
\ni (b_1, \dots, b_k)
\mapsto \sum_{i=1}^{k}d_i b_i \in \mathbb{Z}/e\mathbb{Z}
\end{equation}
is denoted by $H$.
The finite abelian group $G$ acts
on the polydisc $\Delta^n$ by
\begin{equation}
(b_1, \dots, b_k) \cdot y_i
=
\begin{cases}
\exp(2\pi\sqrt{-1} b_i/c_i)y_i &\quad \text{for $1 \le i \le k$} \\
y_i &\quad \text{for $k+1 \le i \le n$},
\end{cases}
\end{equation} 
where $(y_1, \dots, y_n)$ is the coordinate of $\Delta^n$.
Then
$\widetilde{U} \simeq \Delta^n/H$
and $\widetilde{f}^*t=y_1 \cdots y_k$,
where $t$ is the coordinate of $\Delta$.
Note that $y_1 \cdots y_k$ is $H$-invariant.
Moreover, $U=\pi(\widetilde{U})$ is an open subset of $X$,
and we also have
$U \simeq \Delta^n/G$
and $f^*t=(y_1 \cdots y_k)^N$.
Here we note that $(y_1 \cdots y_k)^N$ is $G$-invariant
because $N \in e\mathbb{Z}$.
The $G$-invariant functions
$y_1^{c_1}, \dots, y_k^{c_k}, y_{k+1}, \dots, y_n$
give us a coordinate on $U$.

From the local description above,
$\widetilde{X}$ is trivially a $V$-manifold.
We can easily see that
$E_i$ is a reduced Cartier divisor
on $X \setminus \bigcup_{j \not= i}E_j$.
Moreover, $E_i$ is locally irreducible at any point
because $\pi(E_i)$ is an irreducible component of $\Supp f^{-1}(0)$
and because $\Supp f^{-1}(0)$ is a simple normal crossing divisor
on $X$.
\end{say}

\begin{say}\label{c-say6.3}
In the situation \ref{c-say6.1},
the log structure on $\widetilde{X}$
associated to the effective divisor $E$
is denoted by $\mathcal{M}$,
that is,
\begin{equation}
\mathcal{M}:=
\mathcal{O}_{\widetilde{X}}
\cap
j_*\mathcal{O}^*_{\widetilde{X} \setminus E}
\end{equation}
in $j_*\mathcal{O}_{\widetilde{X} \setminus E}$,
where $j$ denotes the open immersion
$\widetilde{X} \setminus E \hookrightarrow \widetilde{X}$.
The abelian sheaf
associated to the monoid sheaf $\mathcal{M}$
is denoted by $\mathcal{M}\gp$.
By using the fact that $E_i$ is locally irreducible,
a morphism of monoid sheaves
$\mathcal{M} \to (a_i)_*\mathbb{N}_{E_i}$
can be defined by
\begin{equation}
\label{c-eq6.1}
\mathcal{M}=
\mathcal{O}_{\widetilde{X}}
\cap
j_*\mathcal{O}^*_{\widetilde{X} \setminus E}
\ni a
\mapsto
\ord_{E_i}(a) \in (a_i)_*\mathbb{N}_{E_i}
\end{equation}
for any $i$,
where $\ord_{E_i}$
denotes the vanishing order of a holomorphic function on $\widetilde{X}$
along the divisor $E_i$.
The direct sum of the morphisms
\eqref{c-eq6.1} for all $i$ 
induces a morphism
\begin{equation}
\label{c-eq6.2}
\mathcal{M}\gp \to \bigoplus_{i=1}^l(a_i)_*\mathbb{Z}_{E_i},
\end{equation}
which fits in an exact sequence
\begin{equation}
\label{c-eq6.3}
0 \to \mathcal{O}^*_{\widetilde{X}} \to \mathcal{M}\gp
\to \bigoplus_{i=1}^l(a_i)_*\mathbb{Z}_{E_i}
\end{equation}
by definition.
\end{say}

The following is a key lemma for the construction of $(A_{\mathbb{Q}}, W)$.

\begin{lem}
\label{c-lem6.4}
We obtain the exact sequence
\begin{equation}
0 \to \mathcal{O}^*_{\widetilde{X}} \otimes_{\mathbb{Z}}\mathbb{Q}
\to \mathcal{M}\gp \otimes_{\mathbb{Z}}\mathbb{Q}
\to \bigoplus_{i=1}^l(a_i)_*\mathbb{Q}_{E_i}
\to 0
\end{equation}
by tensoring $\mathbb{Q}$ to \eqref{c-eq6.3}.
\end{lem}
\begin{proof}
We may work in the local situation
described in \ref{c-say6.2}.
Since $y_i^{c_i}$ is $H$-invariant,
it gives us a holomorphic function on $\widetilde{U}$
for $i=1, \dots, k$.
We may assume that $E_i=\Supp \{y^{c_i}=0\}$ for $1 \le i \le k$
and $E_i \cap \widetilde{U}=\emptyset$ for $k+1 \le i \le l$
by changing the indices.
Because $E_i$ is the zero set of
$\widetilde{f}^*t=y_1 \cdots y_k$
on $\widetilde{U} \setminus \bigcup_{j \not=i}(E_j \cap \widetilde{U})$,
the image of $y_i^{c_i} \in \mathcal{M} \subset \mathcal{M}\gp$
by the morphism \eqref{c-eq6.2} is
$(0, \dots, 0, c_i, 0 \dots, 0) \in \bigoplus_{j=1}^l(a_j)_*\mathbb{Z}_{E_j}$,
where $c_i$ is on the $i$-th entry.
Thus we obtain the conclusion.
\end{proof}

\begin{say}\label{c-say6.5}
We briefly recall
the constructions
of the Koszul complexes and related objects. 
For the detail,
see \cite[Sections 1 and 2]{Fujisawa2}
and \cite[\S 4.4 and \S 11.2.6]{Peters-SteenbrinkMHS}
(cf.~\cite{Illusie}, \cite{Steenbrink3} and so on).

A morphism of abelian sheaves
$\mathbf{e} \colon \mathcal{O}_{\widetilde{X}} \to \mathcal{M}\gp$
is defined as the composite of the exponential map
\begin{equation}
\mathcal{O}_{\widetilde{X}} \ni a
\mapsto e^{2\pi\sqrt{-1}a} \in \mathcal{O}^*_{\widetilde{X}}
\end{equation}
and the inclusion
$\mathcal{O}^*_{\widetilde{X}} \hookrightarrow \mathcal{M}\gp$. 
Then the morphism
$\mathbf{e} \otimes \id
\colon \mathcal{O}_{\widetilde{X}}
\simeq \mathcal{O}_{\widetilde{X}} \otimes \mathbb{Q}
\to \mathcal{M}\gp \otimes \mathbb{Q}$
is obtained.
Note that $1 \in \Gamma(X, \mathcal{O}_{\widetilde{X}})$
is contained in the kernel of $\mathbf{e} \otimes \id$.
We set
$\mathcal{M}\gp_{\mathbb{Q}}=\mathcal{M}\gp \otimes \mathbb{Q}$
for short.
For $p \in \mathbb{Z}$,
a $\mathbb{Q}$-sheaf $\kos(\mathcal{M})^p$
on $\widetilde{X}$ is defined by
\begin{equation}
\kos(\mathcal{M})^p
:=\varinjlim_{n}\sym_{\mathbb{Q}}^{n-p}(\mathcal{O}_{\widetilde{X}})
\otimes_{\mathbb{Q}} \bigwedge^p\mathcal{M}\gp_\mathbb{Q},
\end{equation}
where $\sym_{\mathbb{Q}}^{n-p}(\mathcal{O}_{\widetilde{X}})$
denotes the symmetric tensor product of degree $n-p$
of $\mathcal{O}_{\widetilde{X}}$ over $\mathbb{Q}$,
and where the inductive limit is taken over the inductive system
defined by the morphisms
\begin{equation}
\label{eq:16}
\sym_{\mathbb{Q}}^{n-p}(\mathcal{O}_{\widetilde{X}})
\otimes_{\mathbb{Q}} \bigwedge^p\mathcal{M}\gp_\mathbb{Q}
\ni a \otimes b \mapsto
(1 \cdot a) \otimes b
\in \sym_{\mathbb{Q}}^{n+1-p}(\mathcal{O}_{\widetilde{X}})
\otimes_{\mathbb{Q}} \bigwedge^p\mathcal{M}\gp_\mathbb{Q}
\end{equation}
for all $n \ge p$.
A morphism of $\mathbb{Q}$-sheaves
\begin{equation}
\sym_{\mathbb{Q}}^{n-p}(\mathcal{O}_{\widetilde{X}})
\otimes_{\mathbb{Q}} \bigwedge^p\mathcal{M}\gp_\mathbb{Q}
\to
\sym_{\mathbb{Q}}^{n-p-1}(\mathcal{O}_{\widetilde{X}})
\otimes_{\mathbb{Q}} \bigwedge^{p+1}\mathcal{M}\gp_\mathbb{Q}
\end{equation}
is defined by
\begin{equation}
\label{eq:17}
a_1^{n_1} \cdots a_k^{n_k} \otimes b
\mapsto
\sum_{j=1}^{k}n_ja_1^{n_1} \cdots a_j^{n_j-1} \cdots a_k^{n_k}
\otimes (\mathbf{e} \otimes \id)(a_j) \wedge b
\end{equation}
where $n_1, \dots, n_k$ are positive integers
with $n_1+\cdots+n_k=n-p$.
These morphisms form a morphism of inductive systems 
defined by \eqref{eq:16} for $p$ and $p+1$
because $1 \in \Gamma(X, \mathcal{O}_{\widetilde{X}})$
is contained in the kernel of $\mathbf{e} \otimes \id$.
Then a morphism of $\mathbb{Q}$-sheaves
\begin{equation}
d \colon \kos(\mathcal{M})^p \to \kos(\mathcal{M})^{p+1}
\end{equation}
is induced.
We can easily see the equality $d^2=0$.
Thus we obtain a complex of $\mathbb{Q}$-sheaves
$\kos(\mathcal{M})$ on $\widetilde{X}$.
Replacing $\mathcal{M}\gp$ by $\mathcal{O}^*_{\widetilde{X}}$,
we obtain a complex of $\mathbb{Q}$-sheaves
$\kos(\mathcal{O}^*_{\widetilde{X}})$.

We set
\begin{equation}
W_m(\sym_{\mathbb{Q}}^{n-p}(\mathcal{O}_{\widetilde{X}})
\otimes_{\mathbb{Q}} \bigwedge^p \mathcal{M}\gp_{\mathbb{Q}})
=\sym_{\mathbb{Q}}^{n-p}(\mathcal{O}_{\widetilde{X}})
\otimes_{\mathbb{Q}}
\bigwedge^{p-m}(\mathcal{O}^*_{\widetilde{X}} \otimes \mathbb{Q})
\otimes_{\mathbb{Q}}
\bigwedge^m \mathcal{M}\gp_{\mathbb{Q}},
\end{equation}
which is a $\mathbb{Q}$-subsheaf of
$\sym_{\mathbb{Q}}^{n-p}(\mathcal{O}_{\widetilde{X}})
\otimes_{\mathbb{Q}} \bigwedge^p \mathcal{M}\gp_{\mathbb{Q}}$
for every $m \in \mathbb{Z}$.
Since the morphism \eqref{eq:16} trivially preserves
$W_m$ on the both sides,
a subsheaf $W_m\kos(\mathcal{M})^p$ of $\kos(\mathcal{M})^p$
is obtained by
\begin{equation}
W_m\kos(\mathcal{M})^p
=\varinjlim_{n}W_m(\sym_{\mathbb{Q}}^{n-p}(\mathcal{O}_{\widetilde{X}})
\otimes_{\mathbb{Q}} \bigwedge^p \mathcal{M}\gp_{\mathbb{Q}})
\end{equation}
for every $m$.
It can be easily checked
that they define a finite increasing filtration $W$
on the complex $\kos(\mathcal{M})$.

The singular locus of $\widetilde{X}$
is denoted by $\sing(\widetilde{X})$
and the smooth locus $\widetilde{X}\sm$
is defined by $\widetilde{X}\sm=\widetilde{X} \setminus \sing(\widetilde{X})$.
Note that the restriction $E \cap \widetilde{X}\sm$
of $E$ to $\widetilde{X}\sm$
is a simple normal crossing divisor on $\widetilde{X}\sm$.
Then we have a morphism of monoid sheaves
\begin{equation}
\dlog \colon
\mathcal{M}|_{\widetilde{X}\sm}
\to \Omega^1_{\widetilde{X}\sm}(\log E \cap \widetilde{X}\sm)
\end{equation}
defined by
$\dlog (a):=a^{-1}da$ for
$a \in \mathcal{M}|_{\widetilde{X}\sm}
\subset \mathcal{O}_{\widetilde{X}\sm}$.
Thus we obtain a morphism
\begin{equation}
\mathcal{M}
\to \widetilde{\Omega}^1_{\widetilde{X}}(\log E)
\end{equation}
denoted by the same letter $\dlog$ as above.
This morphism induces a morphism of $\mathbb{Q}$-sheaves
\begin{equation}
\bigwedge^p\mathcal{M}\gp_{\mathbb{Q}}
\to
\widetilde{\Omega}^p_{\widetilde{X}}(\log E)
\end{equation}
denoted by $\bigwedge^p \dlog$ for every $p$.
A morphism of $\mathbb{Q}$-sheaves
\begin{equation}
\sym^{n-p}_{\mathbb{Q}}(\mathcal{O}_{\widetilde{X}})
\otimes_{\mathbb{Q}} \bigwedge^p\mathcal{M}\gp_{\mathbb{Q}}
\to
\widetilde{\Omega}^p_{\widetilde{X}}(\log E),
\end{equation}
defined by
\begin{equation}
a_1^{n_1} \cdots a_k^{n_k} \otimes b
\mapsto
(2\pi\sqrt{-1})^{-p}a_1^{n_1} \cdots a_k^{n_k}\bigwedge^p\dlog (b)
\end{equation}
for positive integers $n_1, \dots, n_k$ with $n_1+\dots+n_k=n-p$,
is compatible with the morphisms \eqref{eq:16} and \eqref{eq:17}.
Therefore we have a morphism of complexes of $\mathbb{Q}$-sheaves
\begin{equation}
\kos(\mathcal{M}) \to \widetilde{\Omega}_{\widetilde{X}}(\log E),
\end{equation}
which is denoted by $\psi$ as in \cite[(2.4)]{Fujisawa2}.
It can be easily seen that
the morphism $\psi$ preserves the filtration $W$
on the both sides.

The global section
$\widetilde{f}^*t \in \Gamma(\widetilde{X}, \mathcal{M})$
defines a morphism of complexes
\begin{equation}
(\widetilde{f}^*t) \wedge
\colon
\kos(\mathcal{M}) \to \kos(\mathcal{M})[1],
\end{equation}
which sends $W_m\kos(\mathcal{M})^n$
to $W_{m+1}\kos(\mathcal{M})^{n+1}$ 
as in \cite[(1.11) and (1.12)]{Fujisawa2}.
It can be easily checked
that the diagram
\begin{equation}
\label{c-eq6.4}
\xymatrix{ 
\kos(\mathcal{M}) \ar[d]_-{(\widetilde{f}^*t)\wedge} 
\ar[rr]^-{\psi}&& \widetilde{\Omega}_{\widetilde{X}}(\log E)
\ar[d]^-{\theta \wedge} \\
\kos(\mathcal{M})[1]
\ar[rr]_-{(2\pi\sqrt{-1})\psi}&&
\widetilde{\Omega}_{\widetilde{X}}(\log E)[1]
}
\end{equation} 
is commutative,
where
$\theta=\widetilde{f}^*(dt/t) \in \widetilde{\Omega}_{\widetilde{X}}^1(\log E)$.
\end{say}

For $\kos(\mathcal{M})$ and $\psi$ above,
we have the following lemmas.

\begin{lem}
\label{c-lem6.6}
In the situation above,
we set
\begin{equation}
E^{(m)}=\coprod_{1 \le i_1 < \dots < i_m \le l}E_{i_1} \cap \cdots \cap E_{i_m}
\end{equation}
for $m \in \mathbb{Z}_{> 0}$.
Moreover, we set $E^{(0)}=\widetilde{X}$.
The natural morphism $E^{(m)} \to \widetilde{X}$ is denoted
by $a_m$
for $m \in \mathbb{Z}_{\ge 0}$.
Then there exists a quasi-isomorphism
\begin{equation}
(a_m)_*\mathbb{Q}_{E^{(m)}}[-m]
\to \Gr_m^W\kos(\mathcal{M})
\end{equation}
for all $m \in \mathbb{Z}$.
\end{lem}
\begin{proof}
We have an isomorphism
\begin{equation}
\bigwedge^m
(\mathcal{M}\gp \otimes \mathbb{Q}
/\mathcal{O}_{\widetilde{X}}^* \otimes \mathbb{Q})
\otimes
\kos(\mathcal{O}_{\widetilde{X}}^*)[-m]
\simeq
\Gr_m^W\kos(\mathcal{M})
\end{equation}
by \cite[Proposition 1.10]{Fujisawa2},
and a quasi-isomorphism
$\mathbb{Q}_{\widetilde{X}} \to \kos(\mathcal{O}_{\widetilde{X}}^*)$
by \cite[Corollary 1.15]{Fujisawa2}.
Therefore we obtain the conclusion by Lemma \ref{c-lem6.4}.
\end{proof}

\begin{lem}
\label{c-lem6.7}
In the situation above,
we have the commutative diagram
\begin{equation}
\label{c-eq6.5}
\xymatrix{
(a_m)_*\mathbb{Q}_{E^{(m)}}[-m]\ar[d]
\ar[rrr]^-{(2\pi\sqrt{-1})^{-m}\iota[-m]}&&&
(a_m)_*\widetilde{\Omega}_{E^{(m)}}[-m]\ar[d]^-{\simeq} \\
\Gr_m^W\kos(\mathcal{M})
\ar[rrr]_-{\Gr_m^W\psi}
&&&\Gr_m^W\widetilde{\Omega}_{\widetilde{X}}(\log E)
}
\end{equation} 
where $\iota$ is the natural morphism
induced from the inclusion 
$\mathbb{Q}_{E^{(m)}} \to \mathcal{O}_{E^{(m)}}$,
the left vertical arrow is
the quasi-isomorphism in Lemma \ref{c-lem6.6},
and the right vertical arrow is the 
inverse of the residue isomorphism
in \cite[(1.18) Definition and (1.19) Lemma]{steenbrink2} 
$($see also \cite[3.5]{DeligneED}$)$.
In particular,
the morphism
\begin{equation}
\kos(\mathcal{M}) \otimes \mathbb{C}
\to \widetilde{\Omega}_{\widetilde{X}}(\log E)
\end{equation}
induced by $\psi$ is a filtered quasi-isomorphism
with respect to $W$ on the both sides.
\end{lem}
\begin{proof}
The commutativity of the diagram \eqref{c-eq6.5}
can be checked by the direct computation
from the definition of $\psi$
(cf.~\cite[(2.4)]{Fujisawa2}).
Then the latter conclusion follows from
\cite[(1.9) Corollary]{steenbrink2}.
\end{proof}

Once we obtain these two lemmas,
it is more or less clear that the construction,
parallel to $A_{\mathbb{C}}$
in \cite[(4.14) and (4.17)]{steenbrink1}
and \cite[(2.8)]{steenbrink2},
works for $A_{\mathbb{Q}}$.

\begin{defn}\label{c-def6.8} 
In the situation \ref{c-say6.1},
a bifiltered complex of $\mathbb{C}$-sheaves
$(A_{\mathbb{C}},W,F)$ on $\widetilde{X}$
is defined by
\begin{gather}
A^n_{\mathbb{C}}
:=\bigoplus_{q \ge 0}
\widetilde{\Omega}^{n+1}_{\widetilde{X}}(\log E)/
W_q\widetilde{\Omega}^{n+1}_{\widetilde{X}}(\log E), \\
W_mA^n_{\mathbb{C}}
:=\bigoplus_{q \ge 0}
W_{m+2q+1}\widetilde{\Omega}^{n+1}_{\widetilde{X}}(\log E)/
W_q\widetilde{\Omega}^{n+1}_{\widetilde{X}}(\log E), \\
F^pA^n_{\mathbb{C}}
:=
\bigoplus_{0 \le q \le n-p}
\widetilde{\Omega}_{\widetilde{X}}^{n+1}(\log E)
/W_q\widetilde{\Omega}_{\widetilde{X}}^{n+1}(\log E)
\end{gather}
with the differential
$-d-\theta \wedge$,
where $d$ denotes the differential of the complex
$\widetilde{\Omega}_{\widetilde{X}}(\log E)$
as in \cite[(4.17)]{steenbrink1}. 
Similarly, a filtered complex of $\mathbb{Q}$-sheaves
$(A_{\mathbb{Q}}, W)$ on $\widetilde{X}$
is defined by
\begin{gather}
A^n_{\mathbb{Q}}
:=\bigoplus_{q \ge 0}\kos(\mathcal{M})^{n+1}/W_q\kos(\mathcal{M})^{n+1} \\
W_mA^n_{\mathbb{Q}}
:=\bigoplus_{q \ge 0}W_{m+2q+1}\kos(\mathcal{M})^{n+1}/W_q\kos(\mathcal{M})^{n+1}
\end{gather}
with the differential
$-d-(\widetilde{f}^*t) \wedge$,
where $d$ denotes the differential of the complex $\kos(\mathcal{M})$.
The direct sum of the morphisms of $\mathbb{Q}$-sheaves
\begin{equation}
(2\pi\sqrt{-1})^{q+1}\psi
\colon
\kos(\mathcal{M})^{n+1}/W_q\kos(\mathcal{M})^{n+1}
\to
\widetilde{\Omega}_{\widetilde{X}}^{n+1}(\log E)
/W_q\widetilde{\Omega}_{\widetilde{X}}^{n+1}(\log E)
\end{equation}
gives us a morphism of $\mathbb{Q}$-sheaves
\begin{equation}
A^n_{\mathbb{Q}}
=\bigoplus_{q \ge 0}\kos(\mathcal{M})^{n+1}/W_q\kos(\mathcal{M})^{n+1}
\to
\bigoplus_{q \ge 0}
\widetilde{\Omega}_{\widetilde{X}}^{n+1}(\log E)
/W_q\widetilde{\Omega}_{\widetilde{X}}^{n+1}(\log E)
=A^n_{\mathbb{C}}
\end{equation}
which is compatible with the differentials
by the commutativity of the diagram \eqref{c-eq6.4}.
Thus we obtain a morphism of filtered complexes of $\mathbb{Q}$-sheaves
$\alpha \colon (A_{\mathbb{Q}},W) \to (A_{\mathbb{C}},W)$. 
Note that the supports of $A^n_{\mathbb{C}}$ and $A^n_{\mathbb{Q}}$
are contained in $E$ for every $n$.
Therefore they are the (bi)filtered complexes on $E$. 
\end{defn}

\begin{thm}[{cf.~\cite[(2.8)]{steenbrink2}}]
\label{c-thm6.9}
Let $f \colon X \to \Delta$ be as in \ref{c-say6.1}.
If we assume that $X$ is K\"ahler,
then $((A_{\mathbb{Q}}, W), (A_{\mathbb{C}}, W, F), \alpha)$
is a cohomological $\mathbb{Q}$-mixed Hodge complex on $E$.
\end{thm}
\begin{proof}
By Lemmas \ref{c-lem6.6} and \ref{c-lem6.7},
$(\Gr_m^WA_{\mathbb{Q}}, (\Gr_m^WA_{\mathbb{C}}, F), \Gr_m^W\alpha)$
is identified with the direct sum of
the direct images of
\begin{equation}
(\mathbb{Q}_{E^{(m+2q+1)}}(-m-q)[-m-2q],
(\widetilde{\Omega}_{E^{(m+2q+1)}}[-m-2q], F[-m-q]))
\end{equation}
by the finite morphism $a_{m+2q+1}$
for all $q \ge \max(0, -m)$.
Since $\widetilde{X}$ is an almost K\"ahler $V$-manifold
as in \cite[I.2.5]{Peters-SteenbrinkMHS}
by the assumption for $X$ being K\"ahler,
we obtain the conclusion by Theorem 2.43
of \cite{Peters-SteenbrinkMHS}.
\end{proof}


\end{document}